\documentclass[12pt]{amsart}
\usepackage[margin=1in,letterpaper,portrait]{geometry}
\usepackage{ytableau}
\ytableausetup{notabloids}
\ytableausetup{centertableaux}
\usepackage{latexsym,amsmath,amssymb,verbatim, tikz}
\usepackage{graphicx}
\usepackage[colorlinks]{hyperref}
\usepackage{mathrsfs}
\usepackage{amsfonts}
\usepackage{amsthm,youngtab}
\usepackage{graphicx}
\usepackage{caption}
\usepackage{enumerate,enumitem}
\usepackage{tikz-cd}
\usepackage{rotating}
\usepackage[all]{xy}
\usepackage[titletoc, title]{appendix}
\usepackage{amscd}
\usepackage{float}
\hypersetup{
bookmarksnumbered,
pdfstartview={FitH},
breaklinks=true,
linkcolor=blue,
urlcolor=blue,
citecolor=blue,
bookmarksdepth=2
}	
\usepackage{caption}
\usepackage[utf8]{inputenc}
\usepackage[T1]{fontenc}
\usepackage{adjustbox} 
\usepackage[normalem]{ulem}   
\usepackage{xcolor}           

\theoremstyle{plain}
   
   \newtheorem{theorem}{Theorem}[section]
   \newtheorem{proposition}[theorem]{Proposition}
   
   \newtheorem{lemma}[theorem]{Lemma}
   \newtheorem{corollary}[theorem]{Corollary}

\theoremstyle{definition}

   \newtheorem{remark}[theorem]{Remark}

\numberwithin{equation}{section}

\newcommand{\ZZ}{{\mathbb {Z}}}

\newcommand{\tchi}{{\widetilde{\chi}}}
\newcommand{\td}{{\widetilde{D}}}

\DeclareMathOperator{\Hom}{Hom}

\newcommand\scalemath[2]{\scalebox{#1}{\mbox{\ensuremath{\displaystyle #2}}}}

\newlength{\mysizetiny}
\newlength{\mysizesmall}
\newlength{\mysize}
\newlength{\mysizelarge}
\newcommand{\IZ}{\widehat{I}_{\mathbb{Z}}}
\newcommand{\IxiZ}{\widehat{I}_{\mathbb{Z}}^{ \leq \xi}}
\newcommand{\YZ}{\mathcal{Y}_{\mathbb{Z}}}
\newcommand{\tdxi}{\widetilde{D}_{\xi}}
\newcommand{\Ixi}{I^{\leq \xi}}
\newcommand{\barD}{\overline{D}}
\newcommand{\barDstar}{\overline{D}^{*}}

\newcommand{\CN}{\mathbb{C}[\mathbf{N}]}

\newcommand{\CQ}{\mathcal{C}_Q}
\newcommand{\CZ}{\mathcal{C}_{\mathbb{Z}}}
\newcommand{\CxiZ}{\mathcal{C}_{\mathbb{Z}}^{\leq \xi}}
\newcommand{\YxiZ}{\mathcal{Y}_{\mathbb{Z}}^{\leq \xi}}

\newcommand{\barCQ}{\overline{\mathcal{C}_Q}}

\newcommand{\xistar}{\xi^{*}}
\newcommand{\Aqn}{\mathcal{A}_q(\mathfrak{n})}
\newcommand{\DQ}{\mathcal{D}_Q}

\DeclareMathOperator*{\rep}{Rep}
\DeclareMathOperator*{\ind}{Ind}
\DeclareMathOperator*{\dimvec}{\mathbf{dim}}

\title[Auslander--Reiten combinatorics and $q$-characters]{Auslander--Reiten combinatorics and $q$-characters of representations of affine quantum groups}
\author{\'Elie Casbi and Jian-Rong Li}
\address{\'Elie Casbi: Faculty of Mathematics, University of Vienna, Oskar-Morgenstern Platz 1, 1090 Vienna, Austria.}
\email{elie.casbi@univie.ac.at}
\address{Jian-Rong Li: Faculty of Mathematics, University of Vienna, Oskar-Morgenstern Platz 1, 1090 Vienna, Austria.}
\email{jianrong.li@univie.ac.at}
\date{}

\begin{document}

 \begin{abstract}
  For each simple Lie algebra $\mathfrak{g}$ of simply-laced type, Hernandez and Leclerc introduced a certain category $\CZ$ of finite-dimensional representations of the quantum affine algebra of $\mathfrak{g}$, as well as certain subcategories $\CxiZ$ depending on a choice of height function adapted to an orientation of the Dynkin graph of $\mathfrak{g}$. 
   In our previous work we constructed an algebra homomorphism $\tdxi$ whose domain contains the image of the Grothendieck ring of $\CxiZ$ under the  truncated $q$-character morphism $\tchi_q$ corresponding to $\xi$. We exhibited a close relationship  between the composition of $\tdxi$ with $\tchi_q$ and the morphism $\barD$ recently introduced by Baumann, Kamnitzer and Knutson in their study of the equivariant homology of Mirkovi\'c--Vilonen cycles. In this paper, we extend $\tdxi$ in order to investigate its composition with Frenkel--Reshetikhin's original $q$-character morphism. Our main result consists in proving that the $q$-characters of all standard modules in $\CZ$ lie in the kernel of $\tdxi$. This provides a large family of new non-trivial rational identities suggesting possible geometric interpretations. 
 \end{abstract} 

 \maketitle

\section{Introduction}

Introduced by Drinfeld \cite{Dri86} and Jimbo \cite{Jim85} around 1985, the quantum group $U_q(\mathfrak{g})$ associated to a finite-dimensional simple Lie algebra $\mathfrak{g}$ can be viewed as a deformation of the universal enveloping algebra of $\mathfrak{g}$. 
 In the early 90's, Lusztig \cite{Lus90} constructed a remarkable basis of the quantum coordinate ring $\Aqn$ (dual of the negative half $U_q(\mathfrak{n})$ arising from a triangular decomposition of $U_q(\mathfrak{g})$) called the dual canonical basis. Kashiwara \cite{Kas91} proposed an alternative construction called the upper global basis which was shown  to coincide with Lusztig's dual canonical basis by Grojnowski--Lusztig \cite{grojnowski1993comparison} and Kashiwara--Saito \cite{kashiwara1997geometric}. Evaluating at $q=1$, the dual canonical basis descends to a basis of the ring of regular functions on the Lie group ${\bf N}$ corresponding to $\mathfrak{n}$.  Other remarkable bases of $\CN$ were discovered since then, for instance the dual semicanonical basis constructed by Lusztig \cite{lusztig2000semicanonical}, or the Mirkovi\'c--Vilonen basis arising from the geometric Satake correspondence \cite{MV00}. 
 The study of the combinatorial properties of the elements of these bases  was one of the main motivations at the origin of the foundation of cluster theory by Fomin and Zelevinsky \cite{FZ02}. 

The quantum affine algebra associated to $\mathfrak{g}$, denoted $U_q(\widehat{\mathfrak{g}})$ can be viewed as a affinization of $U_q(\mathfrak{g})$, or as quantization of the universal enveloping algebra of the Kac-Moody algebra $\widehat{\mathfrak{g}}$ associated to $\mathfrak{g}$, see \cite{CP95a}. In \cite{HL10}, Hernandez and Leclerc discovered an intimate connection between the representation theory of quantum affine algebras and cluster theory. They introduced the notion of monoidal categorification of cluster algebras, which opened the way for the study of various cluster structures using representation-theoretic tools. At the heart of this framework lies the idea that certain distinguished short exact sequences in certain well-chosen monoidal categories yield relations in their Grothendieck rings that can be identified with exchange relations associated to mutations for a suitable cluster structure.  Kang--Kashiwara--Kim--Oh \cite{KKKO18} obtained important monoidal categorification results with categories of representations of quiver Hecke algebras by constructing certain non trivial braiding operators ($R$-matrices). These operators give rise to short exact sequences corresponding to the exchange relations for the cluster structure of $\CN$.  They also provide useful (necessary) criteria for the irreducibility of products of simple modules in these categories. Other examples of monoidal categorifications have also been exhibited by Cautis--Williams \cite{CW19} via the coherent Satake category. Hernandez--Leclerc's categories are also crucially involved in the major recent works by Kang--Kashiwara--Kim--Oh-Park \cite{KKK18,KKKO18,KKOP24}. In the original work of Hernandez--Leclerc \cite{HL10}, these monoidal subcategories are constructed for a source--sink orientation. This construction was later generalized to arbitrary height functions in \cite{HL15} and \cite{kashiwara2020monoidal,KKOP24}.

   Beyond the study of cluster monomials, the categories involved in monoidal categorifications of $\CN$ also provide powerful tools to study the dual canonical bases of $\CN$ \cite{Li24} mentioned earlier, as well as dual canonical basis of Grassmannian cluster algebras \cite{CDFL}. For instance, when $\mathfrak{g}$ is of simply-laced type, Hernandez--Leclerc \cite{HL15} introduced a family $\{ \CQ \}$ of (monoidal) categories  of finite-dimensional representations of $U_q(\widehat{\mathfrak{g}})$ indexed by the set of orientations of the Dynkin graph of $\mathfrak{g}$ and showed that for each $Q$, the  complexified Grothendieck ring of $\CQ$ is isomorphic to $\CN$ in such a way that  the natural basis of classes of irreducible representations of $\CQ$ gets identified with the dual canonical basis of $\CN$. These categories $\CQ$ are all monoidal subcategories of a larger  monoidal category $\CZ$ that can be viewed as discrete version of the whole category of finite-dimensional representations of $U_q(\widehat{\mathfrak{g}})$. Frenkel--Reshetikhin's $q$-character morphism \cite[Theorem 3]{FR99} yields an injective ring homomorphism from $K_0(\CZ)$ to the ring of Laurent polynomials in variables $Y_{i,p}$ where $i$ runs over the set of vertices of $Q$ and $p \in i + 2 \mathbb{Z}$. For reasons related to cluster theory, Hernandez--Leclerc consider a slightly modified version $\tchi_q$ of $\chi_q$ called truncated $q$-character, which depends on $Q$ (more precisely on a choice of a height function adapted to $Q$).

 Lusztig \cite{lusztig1983singularities} discovered a deep connection (called the geometric  Satake correspondence) between the representation theory of the Lie group $\mathbf{G}$ of \(\mathfrak{g}\) and the geometry of the affine Grassmannian $Gr_{\mathbf{G}^{\vee}}$ of its Langlands dual $\mathbf{G}^{\vee}$. Using this result together with ideas from \cite{ginzburg1995perverse}, Mirkovi\'c and Vilonen \cite{MV00, mirkovic2007geometric} constructed a remarkable basis $\{ b_z \}$ of $\CN$, called the Mirkovi\'c--Vilonen basis (or simply MV basis), whose elements are  indexed by a family of closed irreducible subvarieties of $Gr_{\mathbf{G}^{\vee}}$ called (stable) Mirkovi\'c--Vilonen cycles. Note that the MV basis has been proved to be biperfect in the sense of \cite{BKK21} (cf. \cite[Theorem 5.2]{BKK21}). These varieties are endowed with a natural action of a maximal torus $\mathbf{T}^{\vee}$ in $\mathbf{G}^{\vee}$ and one can therefore consider the corresponding equivariant homology theory. In their recent work \cite{BKK21}, Baumann--Kamnitzer--Knutson introduced a remarkable algebraic homomorphism $\barD : \CN \longrightarrow \mathbb{C}(\mathfrak{t}^{*})$  where $\mathfrak{t}$ denotes the Lie algebra of $\mathbf{T}^{\vee}$ and proved that the evaluation of $\barD$ on an element $b_Z$ of the MV basis coincides with the equivariant multiplicity at a well-chosen fixed point of the corresponding MV cycle $Z$. The morphism $\barD$ turned out to provide a powerful tool to compare the several distinguished bases of $\CN$ mentioned above. 
   
  The motivation for the present paper stems from our desire to provide an interpretation of $\barD$ from the perspective of Hernandez--Leclerc's  monoidal categorification and investigate further the connections between the representation theory of quantum affine algebras and the geometric Satake correspondence. 
 In our previous work \cite{CasbiLi}, we constructed for each simply-laced type Lie algebra $\mathfrak{g}$ and for each height function $\xi$ on the Dynkin diagram of $\mathfrak{g}$ a morphism $\tdxi : \YxiZ \longrightarrow  \mathbb{C}(\mathfrak{t}^{*})$ where $\YxiZ$ is the subtorus of $\YZ$ generated by the $Y_{i,p}^{\pm 1}, i \in I , p \leq \xi(i)$. The definition of $\tdxi$ relies on Auslander--Reiten combinatorics and involves the coefficients of the inverse quantum Cartan matrix of $\mathfrak{g}$. The ring $\YxiZ$ contains the truncated $q$-characters of all the  representations in Hernandez--Leclerc's category $\CQ$ where $Q$ is the Dynkin quiver given by $\xi$. Our main result from \cite{CasbiLi} consists in proving that the composition of $\tdxi$ with the restriction of the truncated $q$-character morphism to $K_0(\CQ)$ coincides with Baumann--Kamnitzer--Knutson's morphism $\barD$ (more precisely its composition with Hernandez--Leclerc's isomorphism $\CN \simeq K_0(\CQ)$). In other words, $\barD$ can be viewed as the restriction to $\CN$ of $\tdxi \circ \tchi_q$. As a consequence, we were able to extend the results and prove several observations and conjectures made by the first author in \cite{Casbi3}. Moreover,  this allows to extend $\barD$ to the Grothendieck rings of larger subcategories of $\CZ$ also carrying non-trivial cluster structures (see \cite{HL16}). 

This paper is devoted to the  study of the composition of $\tdxi$ with Frenkel--Reshetikhin's original (i.e. non-truncated) $q$-character morphism. Firstly, we extend the map $\tdxi$ to the whole ring $\YZ$ thanks to a periodicity argument. Thus we can study the restriction of $\tdxi$ to the image of $K_0(\CZ)$ under $\chi_q$. A standard module is a tensor product of fundamental modules in a well-chosen order. Our main result can be stated as follows:

 \begin{theorem} \label{thm : main thm intro}
     Let $Q$ be an arbitrary Dynkin quiver and $\xi$ be a height function adapted to $Q$. Let $M$ be a non-trivial standard module in $\CZ$. Then we have that
     $$ \tdxi \left( \chi_q(M) \right) = 0. $$
\end{theorem}

Consequently, the image under $\tdxi$ of the $q$-character of an arbitrary module in $\CZ$ is always constant. Note that this constant is zero in the case of non-trivial simple modules in the subcategory $\CQ$.
The crucial ingredient for the proof of this result consists in establishing a remarkable duality between certain distinguished simple modules in $\CZ$ called Kirillov--Reshetikhin modules, when $\mathfrak{g}$ is of classical type, which we believe is of independent interest. The Kirillov--Reshetikhin modules in $\CZ$ are denoted $X_{i,p}^{(k)}$ with $i \in I$, $p \in \ZZ$ and  $k \in \ZZ_{\ge 1}$ ( see Section \ref{sec : KR modules}). We define a dual orientation $Q^{*}$ given by the height function $\xi^{*}(i) := - \xi(i^{*})$ where $i \mapsto i^{*}$ is the involution of the Dynkin diagram such that $w_0(\alpha_i) = - \alpha_{i^{*}}$. Denoting $\tchi_q^{*}$ the corresponding truncation \`a la Hernandez--Leclerc, we assume $Q$ is the monotonic orientation as chosen in  \cite[Figure 1]{CasbiLi} and show (cf. Theorem~\ref{thm : duality between KRs}) that if $i \in I, 1 \leq k < h$ and $\xi(i)-2h+2 \leq p \leq \xi(i)$, then the following holds
\begin{equation} \label{eq : duality intro}
    \tdxi \left( \tchi_q \left( X_{i,p}^{(k)} \right) \right) = \pm  \left( \psi \circ \widetilde{D}_{\xi^{*}} \right) \left( \tchi_q^{*} \left(  X_{i^{*}, -p-2h+4}^{(h-k)}  \right) \right) 
\end{equation}
where $h$ is the Coxeter number of the Weyl group of $\mathfrak{g}$, the sign on the right hand side can be made explicit in terms of Auslander--Reiten theoretic quantities \cite{AR75}, and $\psi$ denote the automorphism of $\mathbb{C}(\mathfrak{t}^{*})$ given by $\psi(\alpha_i) = \alpha_{i^{*}}$. 
We also use several classical facts from Auslander--Reiten combinatorics as key ingredients in the proof of Theorem~\ref{thm : main thm intro}. In particular, our proof may be viewed as a manifestation of some interesting compatibility between mesh relations and T-systems.

 In the case where $\mathfrak{g}$ is of type $A_n, n \geq 1$, we are able to provide an alternative proof of Theorem~\ref{thm : main thm intro} that does not rely on~\eqref{eq : duality intro}, but rather on Mukhin--Young's path model formulas \cite{MY12}. 

As explained above, the map $\tdxi$ is intimately related to the morphism $\barD$ introduced by Baumann--Kamnitzer--Knutson \cite{BKK21} in their study of the MV basis and of the equivariant homology of the affine Grassmannian. Given this tight connection with the geometric Satake correspondence, it is natural to believe that the statement of Theorem~\ref{thm : main thm intro} may admit some interesting geometric interpretation. 

The paper is organized as follows. Section~\ref{sec : AR} is devoted to some recollection on Auslander--Reiten theory. In particular we investigate the combinatorics of the dual quiver $Q^{*}$ mentioned above. In Section~\ref{sec : reminders HL and Dtilde}, we recall the main features of Hernandez--Leclerc's categorifications and we provide detailed reminders on the map $\tdxi$ defined in our previous work \cite{CasbiLi}. Section~\ref{sec : duality} is devoted to the proof of the duality~\eqref{eq : duality intro} between Kirillov--Reshetikhin modules when $\mathfrak{g}$ is of classical type. Finally, Section~\ref{sec : proof of the main result} contains the final arguments for completing the proof of our main result (Theorem~\ref{thm : main thm intro}). 

\subsection*{Acknowledgments}
The research of E.C. was supported by the RTG NSF grant DMS-1645877.
JR Li is supported by the Austrian Science Fund (FWF): P-34602, Grant DOI: 10.55776/P34602, and PAT 9039323, Grant-DOI 10.55776/PAT9039323, and the National Natural Science Foundation of China (No. 12471023). 

\section{Auslander--Reiten combinatorics}
  \label{sec : AR}

   \subsection{Brief recollection on Auslander--Reiten theory}
    \label{sec : reminders on AR}
Auslander--Reiten theory was introduced by Auslander and Reiten in \cite{AR75}. Let $Q$ be a Dynkin quiver, i.e. an orientation of the Dynkin diagram of a Lie algebra $\mathfrak{g}$ of type $A_n (n \geq 1) , D_n (n \geq 4)$ or $E_n , n=6,7,8$. We use the notation $i \sim j$ to denote that $i$ and $j$ are connected 
by an edge in the underlying graph of $Q$. Let $I$ denote the set of vertices of $Q$ and $n:= \sharp I$. We will denote by $\rep Q$ the category of finite-dimensional representations of $Q$ over $\mathbb{C}$, which we will identify with the category of finite-dimensional right modules over the path algebra $\mathbb{C}Q$ of $Q$. Furthermore we will also denote by $\Delta$ (resp. $\Delta_+$) the set of roots (resp. positive roots) associated to $\mathfrak{g}$, and by $W$ the corresponding Weyl group. 

By Gabriel's theorem \cite{Gab72}, there is a bijection between the set $\Delta_+$ of positive roots of the root system associated with $Q$ and the set of isomorphism classes of indecomposable objects in $\rep Q$. 
Moreover, the dimension vectors of objects in $\rep Q$ can be identified with elements of the positive root cone $\bigoplus_{i \in I} \mathbb{Z}_{\geq 0} \alpha_i$. For each $\beta \in \Delta_+$, we denote by $M_\beta$ the corresponding indecomposable representation. In particular, its dimension vector satisfies $\dim M_\beta = \beta$.

For each $j \in I$, we will respectively denote by $I_j$ (resp. $P_j$) the indecomposable injective (resp. projective) representation of $Q$, whose dimension vectors are given by 
   $$ \dimvec I_j = \sum_{k \rightsquigarrow j} \alpha_k,  \qquad \dimvec P_j = \sum_{l \rightsquigarrow k} \alpha_k, $$
   where the notation $k \rightsquigarrow j$ means that there exists an oriented path in $Q$ starting at $k$ and ending at $j$ (\(k \rightsquigarrow j\) includes the stationary path at vertex \(j\)). 

A reduced expression $\mathbf{i}=(i_1,\dots,i_N)$ of $w_0$ is said to be adapted to $Q$ if for each $k=1,\dots,N$, the vertex $i_k$ is a sink 
in the quiver $Q_{k-1}$ defined inductively by $Q_0=Q$ and $Q_k$ obtained 
from $Q_{k-1}$ by reversing all arrows incident to $i_k$. 
The Coxeter transform $c_Q \in W$ associated to $Q$ is then given by $c_Q := s_{i_1}\cdots s_{i_n}$. We also denote by $h$ the corresponding Coxeter number, i.e. the order of $c_Q$ in $W$ (recall that this number is the same for all Coxeter elements in $W$). 
A height function adapted to $Q$ is a function $\xi : I \longrightarrow \mathbb{Z}$ such that $\xi(j) = \xi(i)-1$ whenever there is an arrow from $i$ to $j$
 in $Q$.

    Let $\DQ := \mathcal{D}^b(\rep Q)$ denote the bounded derived category of $\rep Q$ and $\ind \DQ$ denote the set of indecomposable objects of $\DQ$. The elements of $\ind \DQ$ are of the form $M[m]$ where $M$ is an indecomposable object in $\rep Q$ and $m \in \mathbb{Z}$, so that $\ind \DQ$ is in bijection with $\Delta_+ \times \mathbb{Z}$. We will then set $ \dimvec M[m] := (-1)^m \dimvec M$ for any indecomposable object $M$ in $\rep Q$ and any $m \in \mathbb{Z}$, so that $\dimvec N \in \Delta$ for any object $N \in \ind \DQ$.  
    We denote by $\langle \cdot  , \cdot  \rangle_Q$ the Euler-Ringel form of $Q$, which is the bilinear form given by 
    $$  \forall M,N \in \rep Q, \enspace \langle \dimvec M , \dimvec N \rangle_Q := \dim \Hom_{\rep Q}(M,N) - \dim \mathrm{Ext}^1_{\rep Q}(M,N) . $$
    Using this, it is easy to compute the dimensions of morphism spaces in $\DQ$: namely, if $M$ and $M'$ are indecomposable objects in $\rep Q$ and $m,m' \in \mathbb{Z}$, then 
     \begin{equation} \label{eq : dim of Homs in DbRepQ}
    \dim \mathrm{Hom}_{\DQ}(M[m],M'[m']) =
    \begin{cases}
        (\langle \boldsymbol{\dim}M, \boldsymbol{\dim}M' \rangle )_+& \text{if $m=m'$,} \\
        (\langle \boldsymbol{\dim}M, \boldsymbol{\dim}M' \rangle)_- & \text{if $m' = m+1$,} \\
        0 & \text{otherwise,}
    \end{cases}
    \end{equation}
where $a_+ := \max(a,0)$ and $a_- := - \min(a,0)$ for any integer $a$. 
We denote $i+2\ZZ = \{i+2r: r \in \ZZ\}$ for $i \in I = [n]$, where $n$ is the rank of $\mathfrak{g}$. Let $\IZ$ denote the set of pairs
$$ \IZ := \{ (i,p) , i \in I , p \in i + 2 \mathbb{Z} \}  $$
and let us choose a height function $\xi$ adapted to $Q$. 
Such a choice yields a natural bijection 
\begin{equation} \label{eq : bijection}
(i,p) \in \IZ \longmapsto M_{i,p} \in \ind \DQ 
\end{equation}
given by 
$$ M_{j, \xi(j)} :=  I_j, \enspace j \in I, \qquad M_{i,p \pm 2} := \tau_Q^{ \mp 1} M_{i,p} ,  \enspace (i,p) \in \IZ  $$
where $\tau_Q$ stands for the Auslander-Reiten translate in $\DQ$. Note that one has 
 \begin{equation} \label{eq : dim Hom = 0}
 p>s \enspace  \Rightarrow  \enspace  \dim \mathrm{Hom}_{\DQ}(M_{i,p},M_{j,s}) = 0  
 \end{equation}
for any $(i,p), (j,s) \in \IZ$.
 We then set
$$ \beta_{i,p} := \dimvec M_{i,p} \in \Delta $$
for every $(i,p) \in \IZ$. 
Using standard notations, we will denote by $i \mapsto i^*$ the involution of $I$ given by
\begin{equation} \label{eq : def of istar}
 \alpha_{i^{*}} := - w_0 (\alpha_i). 
 \end{equation}

 
 For each $i \in I$ we denote by $d_i$ the following positive number
 \begin{equation} \label{eq : def of di}
 d_i := \sum_{(j,s) \in \IZ} \dim \Hom_{\DQ} (M_{i, \xi(i)} , M_{j,s}) .
 \end{equation}
 Note that this is well-defined as the right hand side contains only finitely many non zero terms in view of~\eqref{eq : dim of Homs in DbRepQ}. 
 For later purposes, it will be convenient to extend this to all non-negative integers by setting $d_i := 1$ if $i=0$ or $i>n$. Note that for each $i \in I$ we have that 
 $$ \forall p \in i + 2 \mathbb{Z}, \enspace \sum_{(j,s) \in \IZ} \dim \Hom_{\DQ}  (M_{i,p} , M_{j,s}) = d_i  $$
 because $\dim \Hom_{\DQ}(M_{i,p},M_{j,s}) = \dim \Hom_{\DQ}(M_{i,\xi(i)},M_{j,s-p+\xi(i)}) $ for each $(j,s) \in \IZ$. 
 
We then have the following:
  \begin{lemma} \label{lem : cone of morphisms}
 Let $i \in I$ and $p \in i + 2 \mathbb{Z}$. Let $ \mathcal{C}_{i,p} :=\{ (j,s) \in \IZ, p \leq s < p+h \ \} $.
 Then we have that 
 $$ d_i = \sum_{(j,s) \in \mathcal{C}_{i,p}} \langle \beta_{i,p} , \beta_{j,s} \rangle_Q . $$
  \end{lemma}

 \begin{proof}
Let $(i,p) \in \IZ$ and  $(j,s) \in \IZ$ such that $\dim \Hom_{\DQ}(M_{i,p},M_{j,s}) \neq 0$. Then by~\eqref{eq : dim Hom = 0} we have  $s \geq p$. We claim that  $s<p+h$ as well. Indeed, recall that    $S := \tau [1]$ is a Serre functor in $\DQ$ (in the sense of \cite{BonOr01}), i.e. we have bifunctorial isomorphisms 
$$ \Hom_{\DQ}(M,N) \simeq D \Hom_{\DQ}(N,SM) $$
for every objects $M,N$ in $\DQ$ (here $D$ denotes the usual duality of $\mathbb{C}$-vector spaces). This essentially follows from the Auslander-Reiten formula  \cite{AR75}.  Furthermore it is known  (cf. for example [\cite{fujita2021q}, Corollary 2.40]) that $ M_{i,p}[1] \simeq M_{i^*,p+h}$. Thus we have 
$$\dim \Hom_{\DQ}(M_{i,p}, M_{j,s}) = \dim \Hom_{\DQ}(M_{j^* , s-h+2} , M_{i,p})$$
which would be zero if $s>p+h$ in view of~\eqref{eq : dim Hom = 0}. Thus the claim holds, so that $d_i$ can be rewritten as 
$$ d_i = \sum_{(j,s) \in \mathcal{C}_{i,p}} \dim \mathrm{Hom}_{\DQ}(M_{i,p},M_{j,s}) .  $$
In addition to that, the inequalities $p \leq s < p+h$ imply that  $\dim \Hom_{\DQ}(M_{i,p} , M_{j,s}[m]) = 0$ for any non zero integer $m$ and thus $\dim \Hom_{\DQ}(M_{i,p} , M_{j,s}) = \langle \beta_{i,p} , \beta_{j,s} \rangle_Q$ in view of~\eqref{eq : dim of Homs in DbRepQ}, which proves the lemma. 


 \end{proof}

 \begin{corollary} \label{coro di = distar}
     For each $i \in I$  we have that $d_i = d_{i^*}$. 
 \end{corollary}

  \begin{proof}
      We have
\begin{align*}  
    d_i &= \sum_{(j,s) \in \mathcal{C}_{i,p}}  \langle \beta_{i,p} , \beta_{j,s} \rangle_Q = \sum_{(j,s) \in \mathcal{C}_{i,p}} \langle - \beta_{i^* , p+h} , \beta_{j,s} \rangle_Q \\
      &= \sum_{(j,s) \in \mathcal{C}_{i^*,p+h}} \langle - \beta_{i^*,p+h} , \beta_{j^*s-h} \rangle_Q = \sum_{(j,s) \in \mathcal{C}_{i^*,p+h}} \langle \beta_{i^*,p+h} , \beta_{j,s} \rangle_Q = d_{i^*}.  
\end{align*}  
  \end{proof}

 \subsection{Dual Auslander--Reiten theory}

  \label{sec : dual AR}
 
 Fix as above an orientation $Q$ of the Dynkin diagram of $\mathfrak{g}$ and choose a height function $\xi$ adapted to $Q$. We denote by $Q^{*}$ the Dynkin quiver adapted to the height function $\xi^{*}$ given by 
\begin{equation} \label{eq : def of xistar}
 \forall i \in I , \qquad  \xi^{*}(i) := - \xi(i^{*})  
 \end{equation}
 where $i^{*}$ is defined by~\eqref{eq : def of istar}.
For every $i \in I$, we will denote by $I_i^{*}$ (resp. $P_i^{*}$) the indecomposable injective (resp. projective) representation of $Q^{*}$ associated to $i$. We will also denote by $\beta_{i,p}^{*}$ the corresponding root associated to the vertex $(i,p)$ for every $(i,p) \in \IZ$, namely 
  $$ \beta_{i,p}^{*} := \tau_{Q^{*}}^{(\xi^{*}(i)-p)/2} \left(  \dimvec I_i^{*} \right). $$

 \begin{lemma} \label{lem : Euler forms}
 Let $Q$ be a Dynkin quiver. Then  we have that 
 $$ \forall \beta , \gamma \in  \Delta, \qquad  \langle  \beta , \gamma \rangle_Q = \langle  w_0 \gamma , w_0 \beta  \rangle_{Q^{*}} .  $$
  \end{lemma}

\begin{proof}
Recall that \(\alpha_{i}\) is the dimension vector of the simple representation \(S_{i}\).
As everything is bilinear, it suffices to prove the statement in the case where both $\beta$ and $\gamma$ are simple roots. Hence we denote $\beta := \alpha_i$ and $\gamma := \alpha_j$ with $i,j \in I$. Then we have that $-w_0(\beta) = \alpha_{i^{*}}$ and $-w_0(\gamma)=\alpha_{j^{*}}$. Thus we have 
$$  \langle  w_0 \gamma , w_0 \beta  \rangle_{Q^{*}}  = \langle  \alpha_{j^{*}} , \alpha_{i^{*}} \rangle_{Q^{*}}  = \delta_{i^{*} , j^{*}} -  \sharp \{ j^{*} \rightarrow i^{*} \}_{Q^{*}}  . $$
But by definition of $Q^{*}$, we have that 
$$  \sharp \{  k \rightarrow l \}_{Q^{*}}  =  \sharp \{  l^{*} \rightarrow k^{*} \}_Q . $$
Therefore we obtain 
$$  \langle  w_0 \gamma , w_0 \beta  \rangle_{Q^{*}}  = \delta_{i,j} -  \sharp \{ i \rightarrow j \}_{Q}   = \langle \alpha_i , \alpha_j \rangle_Q = \langle \beta , \gamma \rangle_Q  $$
which proves the lemma.
\end{proof}

\begin{lemma}
  For any Dynkin quiver $Q$, we have that $w_0 c_Q w_0 = c_{Q^{*}}^{-1}$. 
\end{lemma}

\begin{proof}
 Fix a reduced expression $\mathbf{i} = (i_1, \ldots , i_N)$ adapted to $Q$, so that a reduced expression for $c_Q$ is given by $(i_1, \ldots , i_n)$ (recall that $n=\sharp I$). 
 As~\eqref{eq : def of istar} implies that $w_0 s_i w_0 = s_{i^{*}}$ for any $i \in I$, we get that  $w_0 c_Q w_0 = s_{i_1^{*}} \cdots s_{i_n^{*}}$. Now, if there is an arrow from $i$ to $j$ in $Q$, then $i$ appears before $j$ in the word $(i_1, \ldots , i_n)$. On the other hand, we have that 
 $$ \xi^{*}(i^{*}) = - \xi(i) = - \xi(j)-1 = \xi^{*}(j^{*})-1 $$
 so there is an arrow from $j^{*}$ to $i^{*}$ in $Q^{*}$, and thus $i^{*}$ has to appear after $j^{*}$ in a reduced expression for $c_{Q^{*}}$. This shows that $ c_{Q^{*}} = s_{i_n^{*}} \cdots s_{i_1^{*}}$ and hence by what precedes $c_{Q^{*}}^{-1} = w_0 c_Q w_0$.
\end{proof}

 \begin{lemma} \label{lem : duality on roots}
 Let $j \in I$ and $s \in j + 2 \mathbb{Z}$. Then we have that 
 $$ w_0 \beta_{j,s} = - \beta_{j, -s-h+2}^{*}. $$
 \end{lemma}

  \begin{proof}
 Let us fix a height function $\xi$ adapted to $Q$. We begin by checking the desired identity in the case  where $s= \xi(j)$, for each $j \in I$. Then we have 
 $$ \beta_{j , \xi(j)} = \dimvec I_j = \sum_{k \rightsquigarrow j} \alpha_k. $$
As explained in the proof of the previous lemma, there is an oriented path from $k$ to $j$ in $Q$ if and only if there is an oriented path from $j^{*}$ to $k^{*}$ in $Q^{*}$. Therefore, we have 
 \begin{align*}
 w_0 \beta_{j,\xi(j)} &= - \sum_{ \text{$k \rightsquigarrow j$ in $Q$}} \alpha_{k^{*}}
 = -  \sum_{ \text{$j^{*} \rightsquigarrow k^{*}$ in $Q^{*}$}} \alpha_{k^{*}} = -  \sum_{ \text{$j^{*} \rightsquigarrow k$ in $Q^{*}$}} \alpha_k = - \dimvec P_{j^{*}}^{*}
 \end{align*}
 where $P_{j^{*}}^{*}$ denotes the indecomposable projective representation of $Q^{*}$ associated to $j^{*}$. By the reminders above, we have that $\dimvec P_{j^{*}}^{*} = \beta_{j, \xi^{*}(j^{*})-h+2}^{*}$. Hence we obtain
 \begin{align*}
 w_0 \beta_{j , \xi(j)} = - \beta_{j, \xi^{*}(j^{*})-h+2}^{*} = - \beta_{j, - \xi(j)-h+2}^{*}
 \end{align*}
 which is the desired statement for $s = \xi(j)$. Fix now $s$ arbitrary. Then by definition we have that $\beta_{j,s} = c_Q^{(\xi(j)-s)/2}(\beta_{j, \xi(j)})$, and hence using the previous lemma, we get
 $$ w_0 \beta_{j,s} = w_0 c_Q^{(\xi(j)-s)/2}(\beta_{j, \xi(j)}) = c_{Q^{*}}^{(s-\xi(j))/2}( w_0 \beta_{j, \xi(j)}) = - c_{Q^{*}}^{(s-\xi(j))/2}(\beta_{j,-\xi(j)-h+2}^{*}) = - \beta_{j, -s-h+2}^{*}
 $$
 which finishes the proof of the lemma. 
 \end{proof}

\subsection{Remark on dual orientations}
 \label{sec : rk on dual orientations}

 We end this section by proving Lemma~\ref{lem : reduced word adapted to Qstar} below. Although we will never use it in the rest of the paper, it sheds light on the motivation for introducing the orientation $Q^{*}$ and therefore we provide a detailed proof for the reader's convenience. First recall the following classical fact. 
  
 \begin{lemma} \label{lem : number of occ}
 Let $\mathbf{i}$ be a reduced word for $w_0$ adapted to an orientation $Q$ of the Dynkin diagram of $\mathfrak{g}$. Then for every $i$, denoting by $n(i)$ the number of occurrences of $i$ in $\mathbf{i}$, we have
  $$ n(i) = \frac{\xi(i)-\xi(i^{*})+h}{2} .  $$
  Consequently, we have that $n(i) + n(i^{*}) = h$ for every $i \in I$.
 \end{lemma}

We can now state and prove the desired statement. 

 \begin{lemma} \label{lem : reduced word adapted to Qstar}
   Fix a reduced word $\mathbf{i} = (i_1, \ldots , i_N)$ of $w_0$ adapted to an orientation $Q$ of the Dynkin diagram of $\mathfrak{g}$ and fix a height function $\xi$ adapted to $Q$.  Let $\mathbf{i}^{*}$ denote the reduced word of $w_0$ given by $\mathbf{i}^{*} := (i_N , \ldots , i_1)$. Then $\mathbf{i}^{*}$ is adapted to the orientation $Q^{*}$. 
 \end{lemma}

 \begin{proof}
  Let us fix $k \in \{1, \ldots, N \}$. We show that $i_k$ is a source of $s_{i_{k+1}} \cdots s_{i_N} Q^{*}$. To do so, we show that 
  $$ s_{i_{k+1}} \cdots s_{i_N} \xi^{*}(i_k) = 1 + s_{i_{k+1}} \cdots s_{i_N} \xi^{*}(j) $$
  for all $j \sim i_k$. By assumption, $\mathbf{i}$ is adapted to $Q$, so $i_k$ is a source of $s_{i_{k-1}} \cdots s_{i_1} Q$. Hence 
  $$ \forall j \sim i_k, \quad s_{i_{k-1}} \cdots s_{i_1} \xi(i_k) = 1+s_{i_{k-1}} \cdots s_{i_1} \xi(j) . $$
  Moreover, we have 
  $$ s_{i_{k-1}} \cdots s_{i_1} \xi(j) =  \xi(j) - 2n_{<k}(j) \quad \text{and} \quad s_{i_{k+1}} \cdots s_{i_N} \xi^{*}(j) = \xi^{*}(j) - 2 n_{>k}(j) $$
  for all $j \in I$. 
  Therefore let us fix $j \sim i_k$  and denote 
  $$ \epsilon := s_{i_{k+1}} \cdots s_{i_N} \xi^{*}(i_k) - s_{i_{k+1}} \cdots s_{i_N} \xi^{*}(j) . $$
  We have 
  \begin{align*}
  \epsilon &= \xi^{*}(i_k) - 2 n_{>k}(i_k) - \xi^{*}(j) + 2 n_{>k}(j) = \xi(j^{*}) - \xi(i_k^{*}) - 2 n_{>k}(i_k) + 2 n_{>k}(j) \\
 &= \xi(j) - \xi(i_k) - 2n(j) + 2n(i_k) - 2 n_{>k}(i_k) + 2 n_{>k}(j) \quad \text{using Lemma~\ref{lem : number of occ},} \\
 &= \xi(j) - 2 n_{\leq k}(i_j) - (\xi(i_k) - 2 n_{\leq k}(i_k)) \\
 &= \xi(j) - 2 n_{<k}(i_j) - (\xi(i_k) - 2 n_{<k}(i_k)-2) \\
 &= s_{i_{k-1}} \cdots s_{i_1} \xi(j) - s_{i_{k-1}} \cdots s_{i_1} \xi(i_k) +2 \\
 &= 1 . 
 \end{align*}
 \end{proof}

\section{Reminders on HL categories and the map \texorpdfstring{$\td$}{td}}
  \label{sec : reminders HL and Dtilde}

   \subsection{Reminders on quantum affine algebras and HL categories}
    \label{sec : HL}

Let $\IZ$ denote the set of pairs
$$ \IZ := \{ (i,p) , i \in I , p \in i + 2 \mathbb{Z} \} $$
 and let $\YZ$ denote the torus
 $$ \YZ := \mathbb{Z}[Y_{i,p}^{\pm 1} , (i,p) \in \IZ]. $$

Following \cite{HL15}, we denote by $\CZ$ the full subcategory of $\mathcal{C}$ generated by the fundamental representations $L(Y_{i,p})$ for $p \in i + 2 \mathbb{Z}$. The simple objects of $\CZ$ are the modules $L(\mathfrak{m})$ where $\mathfrak{m}$ is a monomial in the $Y_{i,p}$. We note that the original construction of such monoidal subcategories for a source--sink orientation was developed in \cite{HL10} and later generalized to arbitrary height functions in \cite{HL15} and \cite{kashiwara2020monoidal, KKOP24}.

Frenkel--Reshetikhin's $q$-character morphism \cite[Theorem 3]{FR99} yields an injective ring homomorphism 
 $$ \chi_q : \CZ \longrightarrow \YZ . $$
 Using standard notation, we set 
 $$ A_{i,p+1} := Y_{i,p}Y_{i,p+2} \prod_{j \sim i} Y_{j,p+1}^{-1} $$
 for each $(i,p) \in \IZ$. It is then a known fact that for a simple object $ M = L(\mathfrak{m})$ in $\CZ$, the renormalized $q$-character of $M$ defined as $\mathfrak{m}^{-1} \chi_q(M)$ is a polynomial in the $A_{i,p+1}^{-1} , (i,p) \in \IZ$ with constant term $1$, see \cite[Conjecture 1]{FR99} and \cite[Theorem 4.1]{FM01}.

  As in the previous section, we fix an orientation $Q$ of the Dynkin diagram of $\mathfrak{g}$ and $\xi$ a height function adapted to $Q$. Then we set
 $$ \IxiZ := \{ (i,p) \in \IZ , p \leq \xi(i) \} $$
 and we denote by $\CxiZ$ the smallest full monoidal subcategory of $\CZ$ containing all the fundamental representations $L(Y_{i,p}) , (i,p) \in \IxiZ$. Given a reduced expression $\mathbf{i}_Q$ of $w_0$ adapted to $Q$, one can define a semi-infinite word denoted $\widehat{\mathbf{i_Q}} := (i_1,i_2, \ldots , )$ by setting $i_{k+mN} := i_k^{*m}$ i.e. the vertex obtained by applying $m$ times the involution $*$ to $i_k$, for every $1 \leq m \leq N$ and every $m \geq 0$. Using standard notations, we set 
$$ k_+ := \min\{l>k \mid i_l = i_k\} \qquad k_- := \max (\{l<k \mid i_l=i_k\} \sqcup \{0\}) .  $$
We also have a bijection 
$$ \varphi : \IxiZ \longrightarrow \mathbb{Z}_{\geq 1} $$
such that $\varphi(i,p-2) = \varphi(i,p)_+$. 
 We also denote by $\CQ$ the smallest full monoidal subcategory of $\CxiZ$ containing all the fundamental representations $L(Y_{i,p}) , (i,p) \in \widehat{I}_Q$ where 
 $$ \widehat{I}_Q := \{ (i,p) \in \IZ, \xi(i^{*})-h \leq  p \leq \xi(i) \} . $$
  The assignment 
  \begin{equation} \label{eq : grading}
      Y_{i,p} \in \YZ \longmapsto \beta_{i,p} \in \Delta
  \end{equation}
  endows $\YZ$ with a grading by the root lattice, such that $A_{j,s+1}$ has a degree $0$ for each $(j,s) \in \IZ$ by the mesh relation following from the fact that there is a distinguished triangle 
  $$ \tau_Q M_{i,p} \longrightarrow \bigoplus_{j \sim i}M_{j,p+1} \longrightarrow M_{i,p} \longrightarrow \tau_Q M_{i,p}[1] $$
  in $\DQ$.  In particular, the $q$-character of any simple module in $\CZ$ is homogeneous with respect to the grading~\eqref{eq : grading}. Moreover, ~\eqref{eq : grading} induces a bijection between $\{ Y_{i,p} , (i,p) \in \widehat{I}_Q \}$ and $\Delta_+$. 
Finally, following \cite{HL16}, we define for each $M \in \CxiZ$ the truncated $q$-character $\tchi_q(M)$ of $M$  as the element of $\YZ$ obtained from $\chi_q(M)$ by dropping all the terms involving variables $Y_{i,p}$ with $p> \xi(i)$. This gives another injective ring homomorphism 
 $$ \tchi_q : \CxiZ \longrightarrow \YxiZ :=\mathbb{Z}[Y_{i,p}^{\pm 1} ,  (i,p) \in \IxiZ] .$$ 
 Truncated $q$-character map is injective because $q$-character map is injective and the $q$-character of a module is determined by its highest $l$-weight and also by its truncated $q$-character \cite{FR99, FM01}.

  \subsection{Kirillov--Reshetikhin modules}
\label{sec : KR modules}

 A distinguished family of simple objects in $\CZ$ will play a central role in this work: the Kirillov--Reshetikhin modules (KR modules). They are the simple modules of the form 
$$ X_{i,p}^{(k)} := L \left( Y_{i,p} \cdots Y_{i,p+2k-2} \right) \quad (i,p) \in \IZ , k \geq 1 . $$
The $q$-characters of the KR modules satisfy the following remarkable identities called T-systems:
$$ \chi_q(X_{i,p}^{(k)}) \chi_q(X_{i,p+2}^{(k)}) = \chi_q(X_{i,p}^{(k+1)}) \chi_q(X_{i,p+2}^{(k)}) + \prod_{j \sim i} \chi_q(X_{j,p+1}^{(k)}) .  $$
Once a height function $\xi$ has been fixed as above, it was shown by Hernandez and Leclerc \cite{HL16} that the Grothendieck ring of $\CxiZ$ carries a cluster algebra structure, and that the classes of the KR modules are cluster variables in this cluster algebra. The T-systems are interpreted as exchange relations associated to a well-chosen mutation sequence, starting with an initial cluster given by the KR module of the form 
 \begin{equation} \label{eq : initial KRs}
 X_{i,p} := X_{i,p}^{( (\xi(i)-p+2)/2 )}, \quad (i,p) \in \IxiZ . 
 \end{equation}

\subsection{The map \texorpdfstring{$\tdxi$}{tdxi}}

 This subsection is devoted to necessary background involving  the constructions and results from our previous work \cite{CasbiLi}. Firstly, recall the coefficients of the inverse quantum Cartan matrix of $\mathfrak{g}$ (assumed to be of simply-laced type). The quantum Cartan matrix $C(z)$ of $\mathfrak{g}$ is defined as 
 $$ C_{i,j}(z) := 
 \begin{cases}
     z+z^{-1}, & \text{if $i=j$,} \\
     -1, & \text{if $i \sim j$,} \\
     0, & \text{otherwise.}
 \end{cases}
 $$
Here $z$ is an indeterminate and the notation $i \sim j$ means that there is an edge relating $i$ and $j$ in the Dynkin graph of $\mathfrak{g}$. We denote by $\widetilde{C}_{i,j}(z)$ the entries of the inverse of $C(z)$. For example, when $\mathfrak{g}$ is of type $D_4$, the series $\widetilde{C}_{i,j}(z)$ are given by
\begin{align*}
\widetilde{C}_{1,1}(z) &= z + z^5 - z^7 - z^{11} + z^{13} + z^{17} - z^{19} + O(z^{21}) \\
\widetilde{C}_{1,2}(z) &= z^2 + z^4 - z^8 - z^{10} + z^{14} + z^{16} - z^{20} + O(z^{22}) \\
\widetilde{C}_{1,3}(z) &= z^3 - z^9 + z^{15} - z^{21} + O(z^{23}) \\
\widetilde{C}_{1,4}(z) &= z^3 - z^9 + z^{15} - z^{21} + O(z^{23}) \\
\widetilde{C}_{2,2}(z) &= z + 2 z^3 + z^5 - z^7 - 2 z^9 - z^{11} + z^{13} + 2 z^{15} + z^{17} - z^{19} + O(z^{21}) \\
\widetilde{C}_{2,3}(z) &= z^2 + z^4 - z^8 - z^{10} + z^{14} + z^{16} - z^{20} + O(z^{22}) \\
\widetilde{C}_{2,4}(z) &= z^2 + z^4 - z^8 - z^{10} + z^{14} + z^{16} - z^{20} + O(z^{22}) \\
\widetilde{C}_{3,3}(z) &= z + z^5 - z^7 - z^{11} + z^{13} + z^{17} - z^{19} + O(z^{21}) \\
\widetilde{C}_{3,4}(z) &= z^3 - z^9 + z^{15} - z^{21} + O(z^{23}) \\
\widetilde{C}_{4,4}(z) &= z + z^5 - z^7 - z^{11} + z^{13} + z^{17} - z^{19} + O(z^{21}). 
\end{align*}

We then define the integers $\widetilde{C}_{i,j}(m)$ for every $i,j \in I$ and $m \geq 0$ by 
$$ \widetilde{C}_{i,j}(z) = \sum_{m \geq 1} \widetilde{C}_{i,j}(m)z^m .$$
The definition is extended to all $m \in \mathbb{Z}$ by setting $\widetilde{C}_{i,j}(m) := 0$ for all $m \leq 0$. The following result is due to Hernandez--Leclerc \cite{HL15}.

\begin{proposition}[ \cite{HL15}, Proposition 2.1] \label{prop : reminders on Ctilde}
Let $(i,p)$ and $(j,s)$ in $\IZ$ and assume  $s \geq p$. Then one has
    $$ \tilde{C}_{i,j}(s-p+1) = \langle \beta_{i,p} ,  \beta_{j,s} \rangle_Q . $$
\end{proposition}

In our previous work \cite{CasbiLi}, we defined a morphism of $\mathbb{C}$-algebras $\tdxi : \mathbb{C} \otimes \YxiZ \longrightarrow \mathbb{C}(\mathfrak{t}^{*})$ by setting 
   $$ \forall (i,p) \in \IxiZ, \enspace \tdxi (Y_{i,p}) := \prod_{  \substack{j \in I \\ s \in \xi(j)- 2 \mathbb{Z}_{\geq 0}} } \beta_{j,s}^{\tilde{C}_{i,j}(s-p-1) - \tilde{C}_{i,j}(s-p+1)} . $$

   \begin{remark} \label{rk : remark on new Dtilde}
 We would like to outline the fact that the above is a slightly modified version of the map $\td$ defined in \cite{CasbiLi}. Here $\beta_{j,s}$ can be a positive or a negative root, whereas the definition of \cite{CasbiLi} does not take the sign into account. In other words, for each $i,p$ the present definition of $\tdxi(Y_{i,p})$ differs by a sign from what it would be with the definition taken in \cite{CasbiLi}.  Note however that both definitions are identical if $\varphi(i,p) \in \{1, \ldots , N\}$, as all the $\beta_{j,s}$ contributing to $\td_{\xi}(Y_{i,p})$ are actually positive roots. Therefore, all the results proved in \cite{CasbiLi} still hold with the present definition of $\td_{\xi}$, as the image of $\CN$ under $\tchi_q$ is contained is the subtorus generated by the $Y_{i,p}$ such that $\varphi(i,p) \in \{1, \ldots , N\}$. 
 \end{remark}

  The following properties can be stated and proved as in \cite{CasbiLi} without any modification.

   \begin{lemma}[\cite{CasbiLi}, Remark 6.7] \label{lem : Dtilde of As}
    For every $(i,p) \in  \IxiZ$, we have that 
    $$ \tdxi(A_{i,p-1}) = \frac{\beta_{i,p}}{\beta_{i,p-2}} .  $$
   \end{lemma}

 The only results from \cite{CasbiLi} that get slightly modified when considering the present definition of $\tdxi$ are Proposition 11.1 and Corollary 11.2. Hence we reproduce detailed proofs below for the reader's convenience. 

  \begin{proposition} \label{prop : prod of Ys equal to pm 1}
For any $(i,p) \in \IxiZ$  we have that 
  $$ \tdxi(Y_{i,p-2h+2} \cdots Y_{i,p}) = (-1)^{d_i} . $$
  \end{proposition}

   \begin{proof}
    From the definition of $\tdxi$, we get 
 $$ \tdxi(Y_{i,p-2h+2} \cdots Y_{i,p}) =  \prod_{(j,s) \in \Ixi} \beta_{j,s}^{ - \tilde{C}_{i,j}(s-p+2h-1) + \tilde{C}_{i,j}(s-p-1)} . $$
  If $(j,s) \in \Ixi$ is such that $s>p+1$, then both $s-p-1$ and $s-p+2h-1$ are positive, and hence $\tilde{C}_{i,j}(s-p+2h-1) = \tilde{C}_{i,j}(s-p-1)$. Similarly, if $s<p-2h+2$, then both $s-p-1$ and $s-p+2h-1$ are negative, and hence $\tilde{C}_{i,j}(s-p+2h-1) = \tilde{C}_{i,j}(s-p-1) = 0$.  Thus only the pairs $(j,s)$ such that $p-2h+2 \leq s \leq p+1$ may have a non-trivial contribution to the product. For such pairs $(j,s)$ we have that $s-p-1 \leq 0$ so that $\tilde{C}_{i,j}(s-p-1)=0$. Hence we can rewrite 
  $$ \tdxi(Y_{i,p-2h+2} \cdots Y_{i,p}) =  \prod_{ \substack{ (j,s) \in \Ixi \\ p-2h+2 \leq s \leq p+1 }} \beta_{j,s}^{- \tilde{C}_{i,j}(s-p+2h-1)} = \prod_{ \substack{ (j,s) \in \Ixi \\ p-2h+2 \leq s \leq p+1}} \beta_{j,s}^{ -\langle \beta_{i,p-2h+2} , \beta_{j,s} \rangle_Q } $$
  by Proposition~\ref{prop : reminders on Ctilde}. We claim that, with the notations of Section~\ref{sec : AR}, we have 
  \begin{equation} \label{eq : set-theoretic eq}
\scalemath{0.92}{  \{ (j,s) \in \Ixi, p-2h+2 \leq s \leq p+1 \mid \langle \beta_{i,p-2h+2} , \beta_{j,s} \rangle_Q \neq 0 \} = \bigsqcup_{(j,s) \in 
 \Ixi \cap \mathcal{C}_{i^*,p-h+2} } \{ (j,s),(j^{*},s-h) \}. }
 \end{equation}
 Let us begin with the direct inclusion. Let $(j,s) \in \Ixi$ such that $(j,s)$ belongs to the left hand side of~\eqref{eq : set-theoretic eq}. If $s \geq p-h+2$ then $(j,s) \in \mathcal{C}_{i^*,p-h+2}$ because $\langle \beta_{i^*,p-h+2} , \beta_{j,s} \rangle_Q = - \langle \beta_{i,p-2h+2} , \beta_{j,s} \rangle_Q \neq 0$ so that $(j,s)$ belongs to the right hand side of~\eqref{eq : set-theoretic eq}. If on the other hand $s < p-h+2$ then setting $j' := j^*$ and $s' := s+h$ we get that $p-h+2 < s' \leq p+1$ so that $(j',s') \in \mathcal{C}_{i^*,p-h+2}$ as $\langle \beta_{i^*,p-h+2} , \beta_{j',s'} \rangle_Q  = \langle \beta_{i,p-2h+2} , \beta_{j,s} \rangle_Q \neq 0$ and $(j,s) = ({j'}^* , s'-h)$. It remains to check that $(j',s') \in \Ixi$ i.e.  $s' \leq \xi(j')$. Assume this is not the case. Then we have $s+h = s' > \xi(j') = \xi(j^*)$ so we have that $\xi(j^*)-h < s \leq \xi(j)$. In other words, $M_{j,s} \in \rep  Q$. We also have 
 $$ \xi(j^*) < s' \leq p + 1 \leq \xi(i)+1 $$
 which implies that $p \geq \xi(j^*) > \xi(i^*)-h$ so that $M_{i,p} \in \rep Q$ as well. Then we have
 $$\langle \beta_{i,p-2h+2} , \beta_{j,s} \rangle_Q = \langle \beta_{i,p+2} , \beta_{j,s} \rangle_Q = -  \langle \beta_{j,s} , \beta_{i,p} \rangle_Q $$
 by the Auslander--Reiten formula. But we have that $\Hom(M_{j,s} , M_{i,p}) = 0$ because $p>s+h-2$ (see the proof of Lemma~\ref{lem : cone of morphisms}) and $\mathrm{Ext}^1(M_{j,s} , M_{i,p}) = 0$ as $p>s$. Therefore we get that $\langle \beta_{i,p-2h+2} , \beta_{j,s} \rangle_Q = 0$ which is a contradiction as $(j,s)$ is assumed to belong to the left hand side of~\eqref{eq : set-theoretic eq}. Hence we have that $s' \leq \xi(j')$ which establishes the direct inclusion.

 
 Conversely, fix $(j,s) \in \Ixi \cap \mathcal{C}_{i^* , p-h+2}$. Then $(j,s)$ belongs to the left hand side of the desired equality as by definition of $\mathcal{C}_{i^*,p-h+2}$ we have $p-h+2 \leq s < p+2$ and $\langle \beta_{i,p-2h+2} , \beta_{j,s} \rangle_Q =-  \langle \beta_{i^*,p-h+2} , \beta_{j,s} \rangle_Q  \neq 0 $. We also have that $(j^*,s-h)$ belongs to the left hand side of~\eqref{eq : set-theoretic eq} as well. Indeed, $s-h \leq \xi(j)-h \leq  \xi(j^*)$ so that $(j^*,s-h) \in \Ixi$, and we also have that $p-2h+2 \leq s-h < p-h+2 \leq p+1$. Finally we also have that $\langle \beta_{i,p-2h+2} , \beta_{j^*,s-h} \rangle_Q = \langle \beta_{i^*,p-h+2} , \beta_{j,s} \rangle_Q \neq 0$ as $(j,s) \in \mathcal{C}_{i^*,p-h+2}$. Thus the converse inclusion holds which proves~\eqref{eq : set-theoretic eq} as claimed. 
We can now write 
  \begin{align*}
      \tdxi(Y_{i,p-2h+2} \cdots Y_{i,p}) &=  \prod_{(j,s) \in \Ixi \cap  \mathcal{C}_{i^*,p-h+2}} \beta_{j,s}^{ - \langle \beta_{i,p-2h+2} , \beta_{j,s} \rangle_Q} (- \beta_{j,s})^{\langle \beta_{i,p-2h+2} , \beta_{j,s} \rangle_Q} \\
      &= (-1)^{ - \sum_{(j,s) \in  \mathcal{C}_{i^*,p-h+2}} \langle \beta_{i^* , p-h+2} , \beta_{j,s} \rangle_Q}  = (-1)^{d_{i^*}} \quad \text{by Lemma~\ref{lem : cone of morphisms},} \\
      &= (-1)^{d_i}  \quad \text{by Corollary~\ref{coro di = distar}.} \end{align*} 
  This concludes the proof of the Proposition. 
   \end{proof} 

  From this, we can derive the following consequences:

   \begin{corollary} \label{coro : periodicity of Ys}
       Let $(i,p) \in \IxiZ$. Then we have 
       $$ \tdxi(Y_{i,p}) = \tdxi (Y_{i,p-2h}) . $$
   \end{corollary}

    \begin{proof}
 Using twice Proposition~\ref{prop : prod of Ys equal to pm 1}, we have 
 \begin{align*}
     \tdxi(Y_{i,p-2h}) &= (-1)^{d_i} \tdxi (Y_{i,p-2h}) \tdxi(Y_{i,p-2h+2} \cdots Y_{i,p}) \\
     &= (-1)^{d_i} \tdxi (Y_{i,p-2h} \cdots Y_{i,p+2}) \tdxi (Y_{i,p}) \\
     &= \tdxi (Y_{i,p}) .
 \end{align*}
    \end{proof}

  Corollary~\ref{coro : periodicity of Ys} has the important consequence that it allows to extend $\tdxi$ to the whole torus $\YZ$ by setting, for any $(i,p) \in \IZ$:
  $$ Y_{i,p} \longmapsto \tdxi(Y_{i,p-2mh}) \quad \text{where $m \geq 0$ is such that $p-2mh \leq \xi(i)$ .}  $$
  Note that Lemma \ref{lem : Dtilde of As} can now be extended to the whole torus $\YZ$. 

   \begin{corollary} \label{coro : Dtilde of A inverse}
    For any $(i,p) \in \IZ$, we have that 
    $$ \tdxi(A_{i,p+1}^{-1}) = \frac{\beta_{i,p}}{\beta_{i,p+2}} . $$
   \end{corollary}

\subsection{Composition with truncated \texorpdfstring{$q$-characters}{q-characters}} \label{sec : Dtilde with truncated q-char}

 Recall that $\tchi_q$ denotes the truncated $q$-character morphism associated to the height function $\xi$ previously fixed.
Recall also the category $\CQ$ from Section~\ref{sec : HL}. 
Hernandez--Leclerc \cite{HL15} constructed an isomorphism of algebras 
$$ \iota_Q : \mathbb{C} \otimes K_0(\CQ) \xrightarrow{\simeq} \CN $$
inducing a bijective correspondence between the classes of simple modules in $\CQ$ and the elements of the dual canonical basis of $\CN$. In view of Remark~\ref{rk : remark on new Dtilde}, we can recall the main result we obtained in  \cite{CasbiLi} using the same notation $\tdxi$.

  \begin{theorem}[{\cite[Theorem 6.1]{CasbiLi}}] \label{thm : thm CasbiLi}
  Let $Q$ be a Dynkin quiver and let $\xi$ be a height function adapted to $Q$. Then  the restriction of the composition $\tdxi \circ \tchi_q$ to $\mathbb{C} \otimes K_0(\CQ)$ coincides with $\barD \circ \iota_Q$ where $\barD$ is the morphism introduced in \cite{BKK21}.  
  \end{theorem}

 The following lemmas will also  be useful later. Recall the notation $X_{i,p}$ from~\eqref{eq : initial KRs}.
   
 \begin{lemma}[\cite{CasbiLi}, Lemma 6.2] \label{lem : Dtilde of initial KRs}
    For any $(i,p) \in \IxiZ$, we have that 
 $$ \left( \tdxi \circ \tchi_q \right) \left( [  X_{i,p} ] \right) = \prod_{(j,s) \in \Ixi} \beta_{j,s}^{ -\tilde{C}_{i,j}(s-p+1)} . $$
  \end{lemma}

  \begin{lemma}[\cite{CasbiLi} Proposition 6.5] \label{lem : relation B}
   Fix a reduced expression $\mathbf{i} = (i_1, \ldots , i_N)$ of $w_0$ adapted to $Q$ and $\widehat{\mathbf{i}}$ denote the corresponding infinite word. Let $(i,p) \in \IxiZ$ and set $t := \varphi(i,p)$. Then denoting $x_t := [X_{i,p}]$, we have
   $$  ( \tdxi \circ \tchi_q) (x_tx_{t_-}) = \beta_{i,p}^{-1}  \left( \tdxi \circ \tchi_q \right) \left( \prod_{ \substack{ u < t < u_+ \\ i_u \sim i_t}} x_u \right) .  $$
  \end{lemma}


   \begin{remark}
  Assume $Q$ is the orientation $Q_0$ considered in \cite{CasbiLi} Section 7, then denoting $\td := \td_{\xi_0}$ where $\xi_0$ is a height function adapted to $Q_0$, the statement of Lemma~\ref{lem : relation B} can be restated as follows:
\begin{equation} \label{eq : relation B for initial KRs}
 \forall (i,p) \in \IxiZ, \enspace   \left( \tdxi \circ \tchi_q  \right) \left( [X_{i,p}] [X_{i,p+2}] \right) = \beta_{i,p}^{-1} \td \left( \prod_{j \sim i} [X_{j,p+1}] \right) .  
 \end{equation}
 \end{remark}




 \section{Duality between Kirillov--Reshetikhin modules}
  \label{sec : duality}

This section is devoted to the key result allowing to prove the main result of this paper. It consists in exhibiting a certain duality between the values taken  on certain Kirillov--Reshetikhin modules by the composition  $\tdxi \circ \tchi_q$ when $\xi$ (as well as the corresponding truncation $\tchi_q$) is respectively adapted to a Dynkin quiver $Q$ and its dual $Q^{*}$. 

In what follows, we will sometimes write $\barD(M)$ instead of $\barD([M])$. 

 \subsection{Preliminaries and statement of the main result}
  \label{sec : prelim and statement of duality}

 Recall that the truncation of the $q$-character heavily depends on the choice of height function. As there will be two height functions ($\xi$ and $\xistar$) involved in the following sections, we will use the notation $\tchi_q^{*}$ to refer to the truncation of the $q$-character morphism given by the height function $\xistar$, in contrast with $\tchi_q$ which is the truncation given by $\xi$. Assume $Q = (Q_{0}^{op})^{*}$ where $Q_0$ is the monotonic orientation defined as in Section 7 in \cite{CasbiLi}  and let us fix a height function $\xi$ adapted to $Q$. Let $\barD , \barDstar : K_0(\CZ) \longrightarrow \mathbb{C}(\mathfrak{t}^{*})$ denote the maps defined as follows:
  $$ \barD (L(Y_{i,p})) := 
  \begin{cases}
   \td_{\xi}( \tchi_q(L(Y_{i,p}))), & \text{if $p \leq \xi(i)$,} \\
   0, & \text{otherwise,}
   \end{cases}
   $$
  $$ \barD^{*} (L(Y_{j,s})) := 
  \begin{cases}
  \left( \psi \circ  \td_{\xi^{*}} \right)( \tchi_q^{*}(L(Y_{j,s}))), & \text{if $s \leq \xi^{*}(j)$,} \\
   0, & \text{otherwise.}
   \end{cases}
   $$
   Here $\psi$ denotes the automorphism of $\mathbb{C}(\mathfrak{t}^{*})$ given by 
 $$ \forall i \in I, \qquad  \psi(\alpha_i) := \alpha_{i^{*}} .  $$
 The reason for this choice of notation is given by Theorem~\ref{thm : thm CasbiLi}: we can view  $\tdxi \circ \tchi_q$ as an extension to a larger domain of Baumann--Kamnitzer--Knutson's morphism $\barD$. Hence by convenience, we still denote by $\barD$ the map defined above on the Grothendieck ring of $\CZ$. 
Recall the notation $d_i$ from (\ref{eq : def of di}). Finally, set $[X_{i,p}^{(0)}] := 1$ for any $(i,p) \in \IZ$. The main result established in this section is the following:
 \begin{theorem} \label{thm : duality between KRs}
  Fix $Q$ and $\xi$ as above and let $i \in I, 0 \leq k < h$ and $p$ such that $\xi(i)-2h+2 \leq p \leq \xi(i)$. Denote $\zeta := (\xi(i)-p-2k+2)/2$. Then we have that
\begin{equation} \label{eq : thm duality between KRs}
    \barD \left( X_{i,p}^{(k)} \right)  = \epsilon_{i,p}^{(k)} \barDstar \left(  X_{i^{*}, -p-2h+4}^{(h-k)}  \right)  
\end{equation}
where $\epsilon_{i,p}^{(k)} \in \{\pm 1\}$ is given by 
$$\epsilon_{i,p}^{(k)}:= 
\begin{cases}
   (-1)^{d_{i+\zeta}-d_{\zeta}} & \text{if $\mathfrak{g}$ is of type $A_n$,} \\
    (-1)^{d_{n-1}+d_n-d_{i+\zeta-n+1}-d_\zeta} & \text{if $\mathfrak{g}$ is of type $D_n$ and $i \leq n-2$,} \\
    (-1)^{d_{(\sigma^{\zeta})(i)}-d_{\zeta}} & \text{if $\mathfrak{g}$ is of type $D_n$ and $i \in \{n-1,n\}$,  }
    \end{cases}
 $$ 
where $\sigma: I \to I$ is the involution of type $D_n$ Dynkin diagram given by $\sigma(i)=i$ for $i \in [n-2]$, $\sigma(n)=n-1$, $\sigma(n-1)=n$. 
\end{theorem}

Given  a simple Lie algebra $\mathfrak{g}$  of type $A_n , n\geq 1$ or $D_n, n \geq 4$ and denoting by $I$ the set of vertices of the corresponding Dynkin diagram, we will say that $i \in I$ is spin if $\mathfrak{g}$ is of type $D_n$ and 
$i \in \{n-1,n\}$ and that $i$ is generic (resp. strictly generic)  if $\mathfrak{g}$ is of type $A_n$ or if $\mathfrak{g}$ is of type $D_n$ and $i \leq n-2$ (resp. $i<n-2$). Furthermore, in all what follows we will use the following notations. Let $\eta : \mathbb{Z} \longrightarrow \mathbb{Z}$ denote the map 
$$\eta  : s \longmapsto -s-2h+4 . $$
Note that $\eta$ induces a bijection 
$$ \eta : \{ s  , \xi(i^{*})-h \geq s > \xi(i)-2h+2 \} \longrightarrow \{ s' ,  \xi^{*}(i^{*}) \geq s' > \xi^{*}(i)-h+2 \} .   $$
 We also define a map $\rho : \{ 0, \ldots , h\} \longrightarrow \{ 0, \ldots , h \}$ by setting
$$ \rho : k \longmapsto h-k . $$
We denote by $\barCQ$ the smallest full monoidal subcategory of $\CZ$ containing  the fundamental representations $L(Y_{i,p}) , i \in I , \xi(i) \geq p \geq \xi(i)-2h+2$. Note that $\barCQ$ contains $\CQ$ as a full monoidal subcategory.
  We moreover introduce an algebra $\mathcal{A}_{\xi}$ which should be regarded as a dual of $K_0(\barCQ)$, in the sense that it is generated by classes of KR modules obtained from those of $\barCQ$ by performing the substitutions $(i,p,k) \mapsto (i^{*}, \eta(p), \rho(k))$. More precisely, we define $\mathcal{A}_{\xi}$ as the sub-algebra of the fraction field of $K_0(\mathcal{C}_{\ZZ})$ generated by the classes of  
  $$ \Sigma = \Sigma_1 \sqcup \Sigma_2 $$
  where 
  $$ \Sigma_1 := \{ [X_{j,s}^{(k)}]^{\pm 1} ,\xistar(j) \geq s > \xistar(j)-2h+2 , k = (\xistar(j)-s+2)/2 \}  $$
 and
 $$\Sigma_2 := \{ [X_{j,s}^{(h)}]^{\pm 1} ,\xistar(j) \geq s > \xistar(j)-2h+2  \}.  $$
 
The strategy of the proof is the following. First, we show that~\eqref{eq : thm duality between KRs} holds for the KR modules belonging to the initial seed, i.e.\ in the case $k \geq 1$ and $p = \xi(i)-2k+2$. This is what is done in Section~\ref{sec : duality for initials KRs}. Then, still considering $p= \xi(i)-2k+2$, we show that the desired equality holds when $k=0$. This amounts to prove that $ \barDstar \left( X_{i^{*}, \xi^{*}(i^{*}) -2h + 4}^{(h)}  \right) = \pm 1$, which will be the purpose of Section~\ref{sec : KR whose deg is h}. Finally, we will conclude using the T-systems in $K_0(\barCQ)$ (resp. $\mathcal{A}_{\xi}$).

\subsection{Classes of KR modules in \texorpdfstring{$\Sigma_1$}{Sigma1}}
 \label{sec : duality for initials KRs}

 In this subsection, we establish the following special case of Theorem~\ref{thm : duality between KRs} which deals with classes of the subset $\Sigma_1$. Note that it is valid for any orientation $Q$ of the Dynkin diagram of $\mathfrak{g}$. 

  \begin{proposition} \label{prop : duality for initial KRs}
 Let $i \in I$, $k \geq 1$ and let $p := \xi(i)-2k+2$. Then one has
   $$ \barD \left(  X_{i,p}^{(k)} \right) = (-1)^{d_i} \barDstar \left( X_{i^{*}, \eta(p)}^{(\rho(k))}  \right) . $$
  \end{proposition}

\begin{proof}
 Recall that by Lemma~\ref{lem : Dtilde of initial KRs} we have
 $$ \barD \left( X_{i,p}^{(k)} \right) = \prod_{ j \in I , s \leq \xi(j)} \beta_{j,s}^{-\tilde{C}_{i,j}(s-p+1)}  $$
with $p := \xi(i)-2k+2$. Moreover, $\tilde{C}_{i,j}(m)=0$ if $m \leq 0$ so that the product runs in fact only on the set $\{ (j,s) , j \in I ,  p \leq s \leq \xi(j) \}$. We split this set into three disjoint subsets as follows:
$$\{ p \leq s \leq \xi(j) \} = \{ p \leq s \leq \xi(j^{*})-h \} \sqcup \{ \xi(j^{*})-h  < s < p+h \}  \sqcup \{ p+h \leq s \leq \xi(j) \} . $$
There is an obvious bijection between the first and the third subset, given by $(j,s) \longmapsto (j^{*},s+h)$. Hence using Proposition~\ref{prop : reminders on Ctilde}, the product over the first and the third subsets can be written as 
\begin{align*}
\prod_{ \substack{j \in I \\ p \leq s \leq \xi(j^{*})-h }} \beta_{j,s}^{-\langle \beta_{i,p} , \beta_{j,s} \rangle} \beta_{j^{*},s+h}^{-\langle \beta_{i,p} , \beta_{j^{*},s+h} \rangle } = \prod_{ \substack{ j \in I \\ p \leq s \leq \xi(j^{*})-h}} \beta_{j,s}^{-\langle \beta_{i,p} , \beta_{j,s} \rangle}  \left( -\beta_{j,s} \right)^{  \langle \beta_{i,p} , \beta_{j,s} \rangle } = (-1)^{d'_i}
\end{align*}
where 
$$ d'_i := \sum_{(j,s) , p \leq s \leq \xi(j^{*})-h}   \langle \beta_{i,p} , \beta_{j,s} \rangle_Q . $$
Thus we have 
$$ \barD \left( X_{i,p}^{(k)} \right) = (-1)^{d'_i} \prod_{ \substack{ j \in I \\ \xi(j^{*})-h<s <p+h}} \beta_{j,s}^{-\tilde{C}_{i,j}(s-p+1)} .  $$
Now we can rewrite this expression using Lemma~\ref{lem : duality on roots} as follows.
\begin{align*}
  \barD \left( X_{i,p}^{(k)} \right) &= (-1)^{d'_i} \prod_{\xi(j^{*})-h<s <p+h} \beta_{j,s}^{ -\langle \beta_{i,p} , \beta_{j,s} \rangle_Q }  \\
  &= (-1)^{d'_i} \prod_{\xi(j^{*})-h<s <p+h} \left( - w_0 \beta_{j,-s-h+2}^{*} \right)^{- \langle \beta_{i,p} , \beta_{j,s} \rangle_Q  }  \\
  &= (-1)^{d'_i} \prod_{-p-2h+2<s<-\xi(j^{*})+2} \left( - w_0 \beta_{j,s}^{*} \right)^{- \langle \beta_{i,p} , \beta_{j,-s-h+2} \rangle_Q  }  \\
  &= (-1)^{d'_i} \prod_{-p-2h+2<s<-\xi(j^{*})+2} \left( - w_0 \beta_{j,s}^{*} \right)^{- \langle  w_0 \beta_{j,-s-h+2} , w_0 \beta_{i,p} \rangle_{Q^{*}} }  \quad \text{by Lemma~\ref{lem : Euler forms},} \\ 
  &= (-1)^{d'_i} \prod_{-p-2h+2<s<-\xi(j^{*})+2} \left( - w_0 \beta_{j,s}^{*} \right)^{- \langle   \beta_{j,s}^{*} , \beta_{i,-p-h+2}^{*} \rangle_{Q^{*}} } \\
  &=  (-1)^{d'_i+d''_i} \prod_{-p-2h+2<s<-\xi(j^{*})+2} \left( w_0 \beta_{j,s}^{*} \right)^{- \langle   \beta_{j,s}^{*} , \beta_{i,-p-h+2}^{*} \rangle_{Q^{*}} }
  \end{align*}
  where 
  $$ d''_i := \sum_{(j,s) , j \in I ,-p-2h+2<s<-\xi(j^{*})+2} \langle   \beta_{j,s}^{*} , \beta_{i,-p-h+2}^{*} \rangle_{Q^{*}} .  $$
  Then, using the Auslander--Reiten formula, we have that 
  $$\langle \beta_{j,s}^{*} , \beta_{i,-p-h+2}^{*} \rangle_{Q^{*}} = - \langle   \beta_{j,s}^{*} , \beta_{i^{*},-p-2h+2}^{*} \rangle_{Q^{*}}  = \langle  \beta_{i^{*},-p-2h+4}^{*} , \beta_{j,s}^{*} \rangle_{Q^{*}} = \langle \beta_{i^{*},\eta(p)}^{*} , \beta_{j,s}^{*} \rangle_{Q^{*}} . $$
  Therefore we obtain 
 \begin{align*}
\barD \left( X_{i,p}^{(k)} \right) &= (-1)^{d'_i+d''_i} \prod_{-p-2h+2<s<-\xi(j^{*})+2} \left( w_0 \beta_{j,s}^{*} \right)^{-\langle \beta_{i^{*},\eta(p)}^{*} , \beta_{j,s}^{*} \rangle_{Q^{*}}} \\
&= (-1)^{d'_i + d''_i} \prod_{\eta(p)  \leq s \leq \xi^{*}(j)} \left( w_0 \beta_{j,s}^{*} \right)^{-\langle \beta_{i^{*},\eta(p)}^{*} , \beta_{j,s}^{*} \rangle_{Q^{*}}} \\
&= (-1)^{d'_i + d''_i} \prod_{\eta(p)  \leq s \leq \xi^{*}(j)} \left( w_0 \beta_{j,s}^{*} \right)^{ -\tilde{C}_{i^{*},j}(s-\eta(p)+1)} 
\end{align*}
using again Proposition~\ref{prop : reminders on Ctilde}. Thus by Lemma~\ref{lem : Dtilde of initial KRs}, this is exactly 
$$\barD \left( X_{i,p}^{(k)} \right) = (-1)^{d'_i + d''_i} \barDstar \left(  X_{i^{*} , \eta(p)}^{(r)}   \right)  $$
where $r = (\xi^{*}(i^{*})-\eta(p)+2)/2$. First of all we have 
$$ r= (-\xi(i)+p+2h-2)/2 = h- (\xi(i)-p+2)/2 = h-k = \rho(k) . $$
Moreover, we claim that $d'_i+d''_i = d_i$. Indeed, the same arguments as above show that $d''_i$ can be rewritten as 
$$ d''_i = \sum_{(j,s) , j \in I , \xi(j^{*})-h < s < p+h} \langle \beta_{i,p} , \beta_{j,s} \rangle_Q .   $$
Thus we have 
$$ d'_i + d''_i = \sum_{(j,s) , j \in I , p \leq s < p+h} \langle \beta_{i,p} , \beta_{j,s} \rangle_Q = d_i $$
 by Lemma~\ref{lem : cone of morphisms} which concludes the proof of the Proposition. 
\end{proof}

\subsection{Classes of KR modules in \texorpdfstring{$\Sigma_2$}{Sigma2}} \label{sec : KR whose deg is h}

  In this subsection we investigate the values of $\barDstar$ on the classes in the subset $\Sigma_2$, assuming from now on that $Q$ is monotonic. We establish the following:

 \begin{proposition} \label{prop : Dbarstar of KRs in Sigma2}
     Let $(j,s)$ such that $[X_{j,s}^{(h)}] \in \Sigma_2$. Then we have that 
     $$ \barDstar \left( X_{j,s}^{(h)} \right) = \epsilon_{j^{*},\eta(s)}^{(0)}, 
     $$
 \end{proposition}
  
For each $i \in I$ and each $j \ge 1$ we denote
 $$ \beta_i^{(j)} := \tau^{j-1}(\gamma_i)  \qquad \text{where $\gamma_i := \boldsymbol{\dim} I_i$.} $$
 We work in the category $\mathcal{C}_{\xi^{*}}$ where $\xi^{*}$ is a height function adapted to $Q^{*} := Q_0^{op}$. For each $i \in I$ and $k \geq 1$ we set 
 $$ P_i^{(k)} := \overline{D}^{*} \left(  X_{i, \xi^{*}(i)-2k+2}^{(k)} \right) . 
 $$
 Thus the $P_i^{(k)}$ are the images under $\barDstar$ of the classes of the KR modules belonging to the initial seed in the category $\mathcal{C}_{\xistar}$. The strategy now consists in using the T-systems relating classes of KR modules, together with the mesh relations after applying $\barDstar$. In order to do so, we need to  introduce the following elements of the root lattice.  We set
$$
\forall i \in I , \forall j \geq 0, \gamma_{i,j} := 
\begin{cases}
    \beta_{\sigma^{n-k}(i)}^{(n-1)}, & \text{if $j \equiv n-1 \mod h$ and $i$ is spin,} \\
    \sum_{l=1}^j \beta_1^{(l)}, & \text{otherwise.}
\end{cases}
$$
We also set $\gamma_{i,j} := 0$ if $j \leq 0$, for every $i \in I$. 
 Then we set 
 $$ \forall i \leq n-1, \forall k \geq 1, \enspace R_i^{(k)} :=  (-1)^{i} \prod_{j=k-i}^{k-1} \gamma_{i,j} . $$
 Note that $R_i^{(k)} = 0$ if $k \leq i$.
 We begin by establishing the following:

  \begin{lemma} \label{lem : formula for KR modules in Kxistar}
  For each $i \in I$ and each $k \geq 1$, we have that 
 $$\overline{D}^{*} \left(  X_{i, \xi^{*}(i)-2k+4}^{(k)} \right) = (-1)^{i} R_i^{(k)} P_i^{(k-1)} . $$
  \end{lemma}

We begin with the following identity, that follows from the mesh relations. 

 \begin{lemma} \label{lem : identity from mesh relations}
   For any $i \in I$ and $k \geq i$, the following identities hold:
   $$ \text{If $i$ is generic, then} \enspace  \gamma_{i,k-i} + \beta_i^{(k)} = \gamma_{i,k} . $$
   $$ \text{If $i$ is spin, then} \enspace \gamma_{i,k-i} \beta_{\sigma^{n-k}(i)}^{(n-1)} + \beta_i^{(k)} \gamma_{i-1,n-1} =\beta_{\sigma^{n-k+1}(i)}^{(n-1)} \gamma_{i,k} . $$
 \end{lemma}

 \begin{proof}
  Assume first that $i$ is generic. Then one can show by a straightforward induction using mesh relations that for any $k \geq i$, one has 
  $$ \beta_i^{(k)} = \sum_{l=k-i+1}^{k} \beta_1^{(l)} .  $$
  Therefore we get 
  $$ \gamma_{i,k-i} + \beta_i^{(k)} = \sum_{l=1}^{k-i} \beta_1^{(l)} + \sum_{l=k-i+1}^{k}  \beta_1^{(l)} = \sum_{l=1}^k  \beta_1^{(l)} = \gamma_{i,k} . $$
  Assume now that $i$ is spin. We will consider the case $i=n-1$, the case $i=n$ being analogous. If $k=h$ then $k-i=n-1$. If $k=i=n-1$ then $k=n-1$.  If $i<k<h$ then $n-1 \in \{ k-i+1 , k-1 \}$. 
 In the last case we have  $\gamma_{i,k-i} = \beta_{k-i}^{(k-i)}$ as $k-i \leq n-2$. Moreover, $i-1$ is generic so $\gamma_{i-1,n-1} = \sum_{l=1}^{n-1} \beta_1^{(l)} = \beta_{n-1}^{(n-1)} + \beta_{n}^{(n-1)}$. 
 We also have that $\gamma_{i,k} = \beta_{h-k}^{(n-1)}$. Therefore we have 
 \begin{align*}
\gamma_{i,k-i} \beta_{\sigma^{n-k}(i)}^{(n-1)} + \beta_i^{(k)} \gamma_{i-1,n-1} &= \beta_{\sigma^{n-k}(i)}^{(n-1)} \left( \beta_{k-i}^{(k-i)} + \beta_i^{(k)} \right) + \beta_i^{(k)} \beta_{\sigma^{n-k+1}(i)}^{(n-1)}  \\
&=\beta_i^{(k)} \beta_{\sigma^{n-k+1}(i)}^{(n-1)} + \beta_{\sigma^{n-k}(i)}^{(n-1)} \left( \beta_i^{(k)} - \beta_{k-i}^{(k)} \right) \quad \text{as $\tau^{h/2}=- 1$,} \\
 &= \beta_i^{(k)} \beta_{\sigma^{n-k+1}(i)}^{(n-1)} + \beta_{\sigma^{n-k}(i)}^{(n-1)} \left( \beta_{h-k}^{(n-1)} - \beta_{\sigma(i)}^{(k)} \right) \quad \text{using mesh relations,} \\
 &= \beta_{\sigma^{n-k}(i)}^{(n-1)} \beta_{h-k}^{(n-1)} + \beta_i^{(k)} \beta_{\sigma^{n-k+1}(i)}^{(n-1)} - \beta_{\sigma^{n-k}(i)}^{(n-1)} \beta_{\sigma(i)}^{(k)} \\
& = \scalemath{0.92}{ \beta_{\sigma^{n-k+1}(i)}^{(n-1)} \beta_{h-k}^{(n-1)} + \beta_{\sigma^{n-k}(i)}^{(n-1)}(\beta_{h-k}^{(n-1)}- \beta_{\sigma(i)}^{(k)}) - \beta_{\sigma^{n-k+1}(i)}^{(n-1)}(\beta_{h-k}^{(n-1)} - \beta_{i}^{(k)}). } 
 \end{align*}
 It remains to observe that 
 $$ \beta_{\sigma^{n-k}(i)}^{(n-1)} =  \beta_{h-k}^{(n-1)} - \beta_i^{(k)}.  $$
 Therefore we get 
 $$\gamma_{i,k-i} \beta_{\sigma^{n-k}(i)}^{(n-1)} + \beta_i^{(k)} \gamma_{i-1,n-1} =\beta_{\sigma^{n-k+1}(i)}^{(n-1)} \beta_{h-k}^{(n-1)} = \beta_{\sigma^{n-k+1}(i)}^{(n-1)} \gamma_{i,k} . $$
 which is the desired identity.   
  \end{proof}

 \begin{corollary} \label{coro : recursion on Rs}
 For any $i \in I$ and $k \geq 1$, we have that
$$ R_i^{(k)} - R_i^{(k+1)} = \beta_i^{(k)} R_{i-1}^{(k)} . $$
 \end{corollary} 
   
\begin{proof}
Recall that $R_j^{(l)} = 0$ if $l \leq j$. Thus if $k<i$, then $R_i^{(k)} = R_i^{(k+1)} = R_{i-1}^{(k)} = 0$ so that the identity is trivial. We can hence assume $k \geq i$. If $i$ is generic, then so is $i-1$ and thus we have that $\gamma_{i,j} = \gamma_{i-1,j}$ for all $j \geq 0$. Hence from the definition of $R_i^{(k)}$, we have 
 $$ R_i^{(k)} - \beta_i^{(k)} R_{i-1}^{(k)} = (-1)^{i} \left(  \prod_{j=k-i}^{k-1} \gamma_{i,j} + \beta_i^{(k)} \prod_{j=k-i+1}^{k-1} \gamma_{i-1,j}  \right) = (-1)^{i}(\gamma_{i,k-i} + \beta_i^{(k)}) \prod_{j=k-i+1}^{k-1} \gamma_{i,j} . $$
 By Lemma~\ref{lem : identity from mesh relations} we obtain 
 $$R_i^{(k)} - \beta_i^{(k)} R_{i-1}^{(k)} = (-1)^{i} \gamma_{i,k} \prod_{j=k-i+1}^{k-1} \gamma_{i,j} = (-1)^{i} \prod_{j=k-i+1}^k \gamma_{i,j} = R_i^{(k+1)}  $$
 which establishes the desired identity. 
 Assume now that $i$ is spin, and without loss of generality $i=n-1$. If $k=h$ then $k-i=n-1$. If $k=i=n-1$ then $k=n-1$.  If $i<k<h$ then $n-1 \in \{ k-i+1 , k-1 \}$. 
 In the last case we have 
 $$ R_{i}^{(k)} - \beta_{i}^{(k)} R_{i-1}^{(k)} = (-1)^{i} \left( \gamma_{i,k-i} \beta_{\sigma^{n-k}(i)}^{(n-1)} + \beta_i^{(k)} \gamma_{i-1,n-1}  \right) \prod_{ \substack{ k-i+1 \leq j \leq k-1 \\ j \neq n-1 \mathrm{mod} h}} \gamma_{i,j} .  $$
Using the previous lemma, we obtain 
$$R_{i}^{(k)} - \beta_{i}^{(k)} R_{i-1}^{(k)} = (-1)^{i} \prod_{j=k-i+1}^{k} \gamma_{i,j}^{(k+1)} = R_i^{(k+1)} $$
which concludes the proof. 
 
\end{proof}

 Now we set $Q_i^{(k)} := R_i^{(k)} P_i^{(k-1)}$ for every $i \in I$ and $k \geq 1$, and we show that the rational functions $Q_i^{(k)}$ and $P_i^{(k)}$ are related by identities corresponding to T-systems. 

 \begin{lemma} \label{lem : T-systems between Q and P}
 If $i$ is strictly generic then we have 
  $$ Q_i^{(k)} P_i^{(k)} = Q_i^{(k+1)} P_i^{(k-1)} + Q_{i-1}^{(k)} P_{i+1}^{(k)} . $$   
  If $i= n-2$ in type $D_n$ then we have 
  $$Q_i^{(k)} P_i^{(k)} = Q_i^{(k+1)} P_i^{(k-1)} + Q_{i-1}^{(k)} P_{n-1}^{(k)} P_n^{(k)} . $$
  If $i$ is spin, then we have 
  $$ Q_i^{(k)} P_i^{(k)} = Q_i^{(k+1)} P_i^{(k-1)} + Q_{i-1}^{(k)}. $$
 \end{lemma}

  \begin{proof}
If $i$ is strictly generic we have
   \begin{align*}
Q_i^{(k+1)} P_i^{(k-1)} + Q_{i-1}^{(k)} P_{i+1}^{(k)} &= R_i^{(k+1)}P_i^{(k)} P_i^{(k-1)} + R_{i-1}^{(k)}P_{i-1}^{(k-1)} P_{i+1}^{(k)} \\
&= R_i^{(k+1)}P_i^{(k)} P_i^{(k-1)} + R_{i-1}^{(k)} \beta_i^{(k)} P_i^{(k)} P_i^{(k-1)}  \quad \text{by~\eqref{eq : relation B for initial KRs},} \\
&= \left( R_i^{(k+1)} + \beta_i^{(k)} R_{i-1}^{(k)} \right) P_i^{(k)}P_i^{(k-1)} \\
&= R_i^{(k)} P_i^{(k)} P_i^{(k-1)} \quad \text{by Corollary~\ref{coro : recursion on Rs}} \\
&= Q_i^{(k)}P_i^{(k)} .
\end{align*}
The two other cases are done similarly. 
  \end{proof}

 \begin{proof}[\textbf{Proof of Lemma~\ref{lem : formula for KR modules in Kxistar}.}]
 
Recall that $R_i^{(k)} = 0$ if $k \leq i$, so in particular, we have that  $Q_i^{(1)} = R_i^{(1)} = 0$ for every $i \in I$. On the other hand, we have that $\barDstar \left( X_{i, \xi^{*}(i)+2}^{(1)} \right) = 0 $  as the truncated $q$-character with respect to $\xi^{*}$ of $L(Y_{i, \xi^{*}(i)+2})$ is trivial. Thus the desired statement trivially holds for every $i \in I$ if $k=1$. 
 Let $k \geq 1$ and assume we have proved that 
 $$\overline{D}^{*}( L(Y_{i, \xi^{*}(i)+2} \cdots Y_{i, \xi^{*}(i)-2k+4})) = Q_i^{(k)} .$$
 for all $i \in I$. 
Let $i \in I$ and $s := \xi^{*}(i)-2k+2$. Note that given our choice of orientation, we have that $\xi^{*}(i-1) = \xi^{*}(i)-1$ so that  $s+1 = \xi^{*}(i-1)-2k+4$. Thus by induction assumption we have 
$$ Q_i^{(k)} = \barDstar \left( [X_{i,s+2}^{(k)}]  \right) \enspace \text{and} \enspace Q_{i-1}^{(k)} = \barDstar \left( [X_{i,s+1}^{(k)}]  \right) .  $$
As $P_j^{(l)} \neq 0$ for any $j \in I$ and $l \geq 0$, Lemma~\ref{lem : T-systems between Q and P} then implies
  \begin{align*}
  Q_i^{(k+1)} &= \frac{1}{P_i^{(k-1)}} \left( Q_i^{(k)}P_i^{(k)} - Q_{i-1}^{(k)}P_{i+1}^{(k)} \right)  \\
  &=  \barDstar \left( \frac{1}{[X_{i,s+2}^{(k-1)}]} \left( [X_{i,s+2}^{(k)}][X_{i,s}^{(k)}] - [X_{i-1,s+1}^{(k)}][X_{i+1,s+1}^{(k)}] \right) \right) = \barDstar \left( [X_{i,s}^{(k+1)}] \right)   
      \end{align*}
  using the T-systems. This concludes the proof of the Proposition. 
 \end{proof}

 From this, we can now derive the two following important consequences. 

  \begin{corollary} \label{cor : relation B for non initial KRs}
      The rational functions $\barDstar \left( X_{j, \xistar(j)-2h+4}^{(k)} \right)$ satisfy identities analogous to~\eqref{eq : relation B for initial KRs}, namely denoting $\widetilde{X_{i,p}} := X_{i,p}^{(\xistar(i)-p+4)/2}$ for each $(i,p) \in \IZ$, we have 
      $$  \barDstar \left( [\widetilde{X_{i,p}}] [\widetilde{X_{i,p+2}}] \right) =  \beta_{i,p+2}^{-1} \prod_{j \sim i} \barDstar \left( [\widetilde{X_{j,p+1}}] \right) .  $$
  \end{corollary}

   \begin{proof}
     Using the previous notations, Lemma~\ref{lem : formula for KR modules in Kxistar} implies that it suffices to check that the following hold: if $i$ is  strictly generic, then 
       $$ \beta_i^{(k-1)} Q_i^{(k)} Q_i^{(k-1)}  = Q_{i-1}^{(k-1)} Q_{i+1}^{(k)}.  $$
       This immediately follows from the definition of $Q_i^{(k)}$ and from~\eqref{eq : relation B for initial KRs}. The other cases are done similarly.
   \end{proof}

  \begin{corollary} \label{cor : vanishing property}
  Let $i \in I$ and $p := \xi^{*}(i)-2h+4$. Then we have 
 $$ \barDstar \left( X_{i,p-2}^{(h+1)} \right)= 0  \quad  \text{and} \quad 
 \barDstar \left( X_{i,p}^{(h)} \right)  = 
  \begin{cases}
      (-1)^{d_{i+1}-d_1} & \text{if $i$ is strictly generic,} \\
      (-1)^{d_{n-1}+d_n-d_1} & \text{if $i=n-2$ in type $D_n$,} \\
      (-1)^{d_{\sigma(i)}-d_1} & \text{if $i$ is spin.}
      \end{cases}$$     
  \end{corollary}

 \begin{proof}
Let $i \in I$. Note that $h+1$ is not equal to $n-1$ modulo $h$ and hence the formula for $\gamma_{i,h}$ is the same whether $i$ is spin or generic. Hence we have that 
$$ \gamma_{i,h} = \sum_{l=1}^{h}  \beta_1^{(l)} = \left( \sum_{l=0}^{h-1} \tau^l \right) (\gamma_1) = 0 $$
as $\tau^h = id$ and it is known that $1$ is not eigenvalue of $\tau$ (cf. \cite[Chapter V, \S6]{BourbakiLieGroupsLieAlgebras4To6}). Hence we get 
$$ Q_i^{(h+1)} = R_i^{(h+1)} P_i^{(h)} = (-1)^{i} P_i^{(h)} \prod_{j=h+1-i}^{h} \gamma_{i,j} = 0 . $$
Consequently, if $i$ is strictly generic then Lemma~\ref{lem : T-systems between Q and P} at the rank $i$ with $k=h$ implies that  
$$ Q_i^{(h)} P_i^{(h)} = Q_{i-1}^{(h)} P_{i+1}^{(h)} .  $$
By Proposition~\ref{prop : prod of Ys equal to pm 1}, this yields
$$ Q_i^{(h)} = (-1)^{d_{i+1}-d_i} Q_{i-1}^{(h)}  $$
and hence a straightforward induction gives 
$$ Q_i^{(h)} = (-1)^{d_{i+1}-d_1}. $$
 Similarly, if $i=n-2$ in type $D_n$, we get 
 $$ Q_i^{(h)}P_i^{(h)} = Q_{i-1}^{(h)}P_{n-1}^{(h)}P_n^{(h)} $$
 and thus using the formula for $i=n-3$ previously obtained we get that 
 $$ Q_{n-2}^{(h)} = (-1)^{d_{n-1}+d_n-d_1} .  $$
 Finally if $i$ is spin, we get 
 $$ Q_i^{(h)} P_i^{(h)} = Q_{n-2}^{(h)}  $$
 which implies using what precedes 
 $$ Q_i^{(h)} = (-1)^{d_{\sigma(i)}-d_1} . $$
Together with Lemma~\ref{lem : formula for KR modules in Kxistar}, we obtain the desired statement. 
 \end{proof}

  \begin{proof}[\textbf{Proof of Proposition~\ref{prop : Dbarstar of KRs in Sigma2}}.]

   Let us sum up what comes out of Lemma~\ref{lem : formula for KR modules in Kxistar}, Corollary~\ref{cor : relation B for non initial KRs} and Corollary~\ref{cor : vanishing property}. Using the fact that the $\barDstar \left( X_{j,\xistar(j)-2h+2}^{(k)} \right)$ satisfy~\eqref{eq : relation B for initial KRs}, we proved that 
   $$ \barDstar \left( X_{j , \xistar(j)-2h+2}^{(h+1)} \right) = 0  $$
   and moreover that the  $\barDstar \left( X_{j , \xistar(j)-2h+4}^{(k)}  \right)$ satisfy analogous relations as~\eqref{eq : relation B for initial KRs} (given by Corollary~\ref{cor : relation B for non initial KRs}). Therefore we can repeat the exact same procedure inductively: using the relations given by Corollary~\ref{cor : relation B for non initial KRs}, mesh relations, and  T-systems  as we did in the proof of Lemma~\ref{lem : formula for KR modules in Kxistar}, we can show that $\barDstar \left( X_{j , \xistar(j)-2h+4}^{(h+1)} \right) = 0$ and deduce from that the values of $\barDstar \left( X_{j , \xistar(j)-2h+6}^{(h)} \right)$ just as what was done to prove Corollary~\ref{cor : vanishing property}. A straightforward induction allows to conclude, which completes the proof of Theorem~\ref{thm : duality between KRs}. 
      
  \end{proof}

\subsection{Duality and T-systems}
\label{sec : duality and T systems}

 Recall from Section~\ref{sec : prelim and statement of duality} the algebra $\mathcal{A}_{\xi}$ as well as the maps $\barD, \barDstar$. Let $x_1, \ldots , x_n$ be formal variables.
 We define a map 
 $$ \mathbf{D} : \mathcal{A}_{\xi} \longrightarrow \mathbb{C}(\alpha_1, \ldots , \alpha_n)(x_1, \ldots , x_n) $$
 as follows. For the classes in $\Sigma_1$  we set 
 $$ \mathbf{D}(X_{j,s}^{(k)}) :=  \barDstar  \left( X_{j , s}^{(k)} \right), \qquad k = (\xistar(j)-s+2)/2 $$
 In order to define the images of $\mathbf{D}$ on the classes of $\Sigma_2$, we need the following piece of notation. For any $j \in I$ and $m \geq 0$, we set 
 $$ z_{j,m} := 
\begin{cases}
    \frac{x_m}{x_{j+m}} & \text{if $\mathfrak{g}$ is of type $A_n$,} \\
    \frac{x_{m+j-n+1}x_m}{x_{n-1}x_n} & \text{if $\mathfrak{g}$ is of type $D_n$ and $j \leq n-2$,} \\
    \frac{x_m}{x_{\sigma^m(j)}} & \text{if $\mathfrak{g}$ is of type $D_n$ and $j \in \{n-1,n\}$,}
    \end{cases}
 $$
where $x_{n+m}=x_{m-1}$ for $m \ge 1$ and $x_0=1$. For $j \in I$, $s, l \in \ZZ$, we define
\begin{align} \label{eq:definition of zeta jsh}
\zeta(j,s,l) := \frac{s-\xistar(j)+2l-2}{2}.
\end{align}
 
 Then we define the images of $\mathbf{D}$ on the classes of $\Sigma_2$ as follows:
 \begin{align} \label{eq:definition of D of Xjsh}
    \mathbf{D}(X_{j,s}^{(h)}) := z_{j,\zeta(j,s,h)}.
 \end{align} 
We begin with a preliminary lemma.

\begin{lemma} \label{lem : boldD not zero}
Let $(i,p)$ such that $\xi(i) \geq p > \xi(i)-2h+2$ and $1 \leq k<h$. Then we have that $[X_{i^{*} , \eta(p)}^{(h-k)}] \in \mathcal{A}_{\xi}$ and moreover we have that 
$$  \mathbf{D} \left( X_{i^{*} , \eta(p)}^{(h-k)} \right) \neq 0   
.  $$
\end{lemma}

\begin{proof}
By definition ${\bf D}(X_{j,s}^{(h)})=z_{j,\zeta(j,s,h)}$. After applying ${\bf D}$ to the equations on $Q^*$ side (see Figures \ref{fig:AR quivers A3}, \ref{fig:AR quivers D5}) which correspond to the equations in the T-system for the first two copies of $\mathcal{C}_Q$ on $Q$ side, we have that for $k \in [1,h - \frac{\xi(i)-p}{2} - 1]$,
\begin{align} \label{eq:apply bD to T-system on Qstar side which correspond to first two copies of CQ on Q side}
{\bf D}(X_{i^*, \eta(p)}^{(h-k)}) {\bf D}(X_{i^*, \eta(p-2)}^{(h-k)}) = {\bf D}(X_{i^*, \eta(p-2)}^{(h-k+1)}) {\bf D}(X_{i^*, \eta(p)}^{(h-k-1)}) + \prod_{j^*:C_{i^*j^*}=-1} {\bf D}(X_{j^*, \eta(p-1)}^{(h-k)}).  
\end{align}
On the right hand side, ${\bf D}(X_{i^*, \eta(p-2)}^{(h-k+1)})$ has degree $1$ in $z_{i^*, \frac{\xi(i)-p}{2}+1}$, ${\bf D}(X_{i^*, \eta(p)}^{(h-k)})$, ${\bf D}(X_{i^*, \eta(p)}^{(h-k-1)})$ and ${\bf D}(X_{j^*, \eta(p-1)}^{(h-k)})$ have degree $0$ in $z_{i^*, \frac{\xi(i)-p}{2}+1}$. Therefore the right hand side is not $0$. It follows that ${\bf D}(X_{i^*, \eta(p)}^{(h-k)}) \ne 0$.

\end{proof}

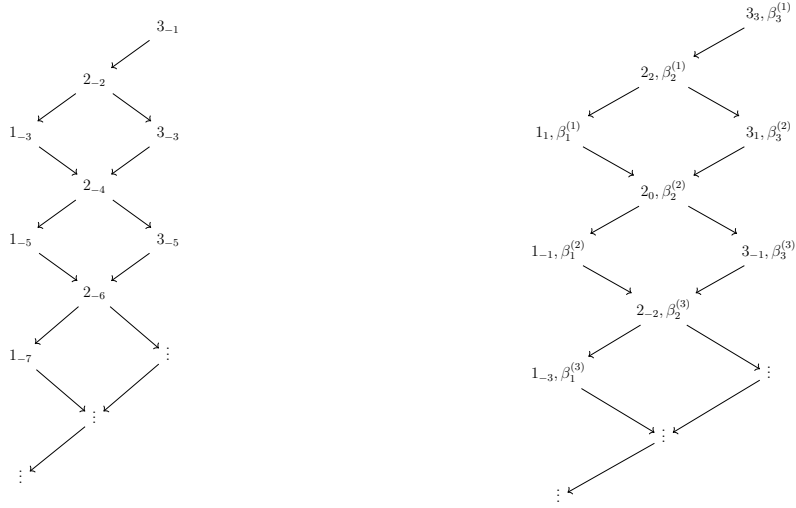
\begin{figure} [t]
     \centering
\begin{minipage}{.45\textwidth}
        \centering
 \adjustbox{scale=0.5}{
 \begin{tikzcd}
	&& {3_{-1}} \\
	& {2_{-2}} \\
	{1_{-3}} && {3_{-3}} \\
	& {2_{-4}} \\
	{1_{-5}} && {3_{-5}} \\
	& {2_{-6}} \\
	{1_{-7}} && \vdots \\
	& \vdots \\
	\vdots
	\arrow[from=1-3, to=2-2]
	\arrow[from=2-2, to=3-1]
	\arrow[from=2-2, to=3-3]
	\arrow[from=3-1, to=4-2]
	\arrow[from=3-3, to=4-2]
	\arrow[from=4-2, to=5-1]
	\arrow[from=4-2, to=5-3]
	\arrow[from=5-1, to=6-2]
	\arrow[from=5-3, to=6-2]
	\arrow[from=6-2, to=7-1]
	\arrow[from=6-2, to=7-3]
	\arrow[from=7-1, to=8-2]
	\arrow[from=7-3, to=8-2]
	\arrow[from=8-2, to=9-1]
\end{tikzcd}
}
\end{minipage}
\begin{minipage}{.45\textwidth}
        \centering
 \adjustbox{scale=0.5}{
\begin{tikzcd}
	&& {3_{3}, \beta_3^{(1)}} \\
	& {2_{2}, \beta_2^{(1)}} \\
	{1_{1}, \beta_1^{(1)}} && {3_{1}, \beta_3^{(2)}} \\
	& {2_{0}, \beta_2^{(2)}} \\
	{1_{-1}, \beta_1^{(2)}} && {3_{-1}, \beta_3^{(3)}} \\
	& {2_{-2}, \beta_2^{(3)}} \\
	{1_{-3}, \beta_1^{(3)}} && \vdots \\
	& \vdots \\
	\vdots
	\arrow[from=1-3, to=2-2]
	\arrow[from=2-2, to=3-1]
	\arrow[from=2-2, to=3-3]
	\arrow[from=3-1, to=4-2]
	\arrow[from=3-3, to=4-2]
	\arrow[from=4-2, to=5-1]
	\arrow[from=4-2, to=5-3]
	\arrow[from=5-1, to=6-2]
	\arrow[from=5-3, to=6-2]
	\arrow[from=6-2, to=7-1]
	\arrow[from=6-2, to=7-3]
	\arrow[from=7-1, to=8-2]
	\arrow[from=7-3, to=8-2]
	\arrow[from=8-2, to=9-1]
\end{tikzcd} }
\end{minipage}
 \caption{Left: Auslander--Reiten quiver for $Q$ of type $A_3$. Right: Auslander--Reiten quiver for $Q^*$ of type $A_3$. We denote $Y_{i,s}$ by $i_s$, $i \in I$, $s \in \ZZ$.}
     \label{fig:AR quivers A3}
 \end{figure}

 \begin{figure} [t]
     \centering
\begin{minipage}{.45\textwidth}
        \centering
 \adjustbox{scale=0.5}{
\begin{tikzcd}
	{1_4} \\
	& {2_3} \\
	{1_2} && {3_2} \\
	& {2_1} && {4_1} && {5_1} \\
	{1_0} && {3_0} \\
	& {2_{-1}} && {4_{-1}} && {5_{-1}} \\
	{1_{-2}} && {3_{-2}} \\
	& {2_{-3}} && {4_{-3}} && {5_{-3}} \\
	{1_{-4}} && {3_{-4}} \\
	& {2_{-5}} && {4_{-5}} && {5_{-5}} \\
	\vdots && {3_{-6}} \\
	& \vdots && {4_{-7}} && {5_{-7}} \\
	&& \vdots \\
	&&& \vdots && \vdots
	\arrow[from=1-1, to=2-2]
	\arrow[from=2-2, to=3-1]
	\arrow[from=2-2, to=3-3]
	\arrow[from=3-1, to=4-2]
	\arrow[from=3-3, to=4-2]
	\arrow[from=3-3, to=4-4]
	\arrow[from=3-3, to=4-6]
	\arrow[from=4-2, to=5-1]
	\arrow[from=4-2, to=5-3]
	\arrow[from=4-4, to=5-3]
	\arrow[from=4-6, to=5-3]
	\arrow[from=5-1, to=6-2]
	\arrow[from=5-3, to=6-2]
	\arrow[from=5-3, to=6-4]
	\arrow[from=5-3, to=6-6]
	\arrow[from=6-2, to=7-1]
	\arrow[from=6-2, to=7-3]
	\arrow[from=6-4, to=7-3]
	\arrow[from=6-6, to=7-3]
	\arrow[from=7-1, to=8-2]
	\arrow[from=7-3, to=8-2]
	\arrow[from=7-3, to=8-4]
	\arrow[from=7-3, to=8-6]
	\arrow[from=8-2, to=9-1]
	\arrow[from=8-2, to=9-3]
	\arrow[from=8-4, to=9-3]
	\arrow[from=8-6, to=9-3]
	\arrow[from=9-1, to=10-2]
	\arrow[from=9-3, to=10-2]
	\arrow[from=9-3, to=10-4]
	\arrow[from=9-3, to=10-6]
	\arrow[from=10-2, to=11-1]
	\arrow[from=10-2, to=11-3]
	\arrow[from=10-4, to=11-3]
	\arrow[from=10-6, to=11-3]
	\arrow[from=11-1, to=12-2]
	\arrow[from=11-3, to=12-2]
	\arrow[from=11-3, to=12-4]
	\arrow[from=11-3, to=12-6]
	\arrow[from=12-2, to=13-3]
	\arrow[from=12-4, to=13-3]
	\arrow[from=12-6, to=13-3]
	\arrow[from=13-3, to=14-4]
	\arrow[from=13-3, to=14-6]
\end{tikzcd}
}
\end{minipage}
\begin{minipage}{.45\textwidth}
        \centering
 \adjustbox{scale=0.5}{
 \begin{tikzcd}
	&&& {4_{-1}, \beta_4^{(1)}} && {5_{-1}, \beta_5^{(1)}} \\
	&& {3_{-2}, \beta_3^{(1)}} \\
	& {2_{-3}, \beta_2^{(1)}} && {4_{-3}, \beta_4^{(2)}} && {5_{-3}, \beta_5^{(2)}} \\
	{1_{-4}, \beta_1^{(1)}} && {3_{-4}, \beta_3^{(2)}} \\
	& {2_{-5}, \beta_2^{(2)}} && {4_{-5}, \beta_4^{(3)}} && {5_{-5}, \beta_5^{(3)}} \\
	{1_{-6}, \beta_1^{(2)}} && {3_{-6}, \beta_3^{(2)}} \\
	& {2_{-7}, \beta_2^{(3)}} && {4_{-7}, \beta_4^{(4)}} && {5_{-7}, \beta_5^{(4)}} \\
	{1_{-8}, \beta_1^{(3)}} && {3_{-8}, \beta_3^{(4)}} \\
	& {2_{-9}, \beta_2^{(4)}} && {4_{-9}, \beta_4^{(5)}} && {5_{-9}, \beta_5^{(5)}} \\
	{1_{-10}, \beta_1^{(4)}} && {3_{-10}, \beta_3^{(5)}} \\
	& {2_{-11}, \beta_2^{(5)}} && \vdots && \vdots \\
	{1_{-12}, \beta_1^{(5)}} && \vdots \\
	& \vdots \\
	\vdots
	\arrow[from=1-4, to=2-3]
	\arrow[from=1-6, to=2-3]
	\arrow[from=2-3, to=3-2]
	\arrow[from=2-3, to=3-4]
	\arrow[from=2-3, to=3-6]
	\arrow[from=3-2, to=4-1]
	\arrow[from=3-2, to=4-3]
	\arrow[from=3-4, to=4-3]
	\arrow[from=3-6, to=4-3]
	\arrow[from=4-1, to=5-2]
	\arrow[from=4-3, to=5-2]
	\arrow[from=4-3, to=5-4]
	\arrow[from=4-3, to=5-6]
	\arrow[from=5-2, to=6-1]
	\arrow[from=5-2, to=6-3]
	\arrow[from=5-4, to=6-3]
	\arrow[from=5-6, to=6-3]
	\arrow[from=6-1, to=7-2]
	\arrow[from=6-3, to=7-2]
	\arrow[from=6-3, to=7-4]
	\arrow[from=6-3, to=7-6]
	\arrow[from=7-2, to=8-1]
	\arrow[from=7-2, to=8-3]
	\arrow[from=7-4, to=8-3]
	\arrow[from=7-6, to=8-3]
	\arrow[from=8-1, to=9-2]
	\arrow[from=8-3, to=9-2]
	\arrow[from=8-3, to=9-4]
	\arrow[from=8-3, to=9-6]
	\arrow[from=9-2, to=10-1]
	\arrow[from=9-2, to=10-3]
	\arrow[from=9-4, to=10-3]
	\arrow[from=9-6, to=10-3]
	\arrow[from=10-1, to=11-2]
	\arrow[from=10-3, to=11-2]
	\arrow[from=10-3, to=11-4]
	\arrow[from=10-3, to=11-6]
	\arrow[from=11-2, to=12-1]
	\arrow[from=11-2, to=12-3]
	\arrow[from=11-4, to=12-3]
	\arrow[from=11-6, to=12-3]
	\arrow[from=12-1, to=13-2]
	\arrow[from=12-3, to=13-2]
	\arrow[from=13-2, to=14-1]
\end{tikzcd}
 }
\end{minipage}
 \caption{Left: Auslander--Reiten for $Q$ of type $D_5$. Right: Auslander--Reiten quiver for $Q^*$ of type $D_5$. We denote $Y_{i,s}$ by $i_s$, $i \in I$, $s \in \ZZ$.}
     \label{fig:AR quivers D5}
 \end{figure}
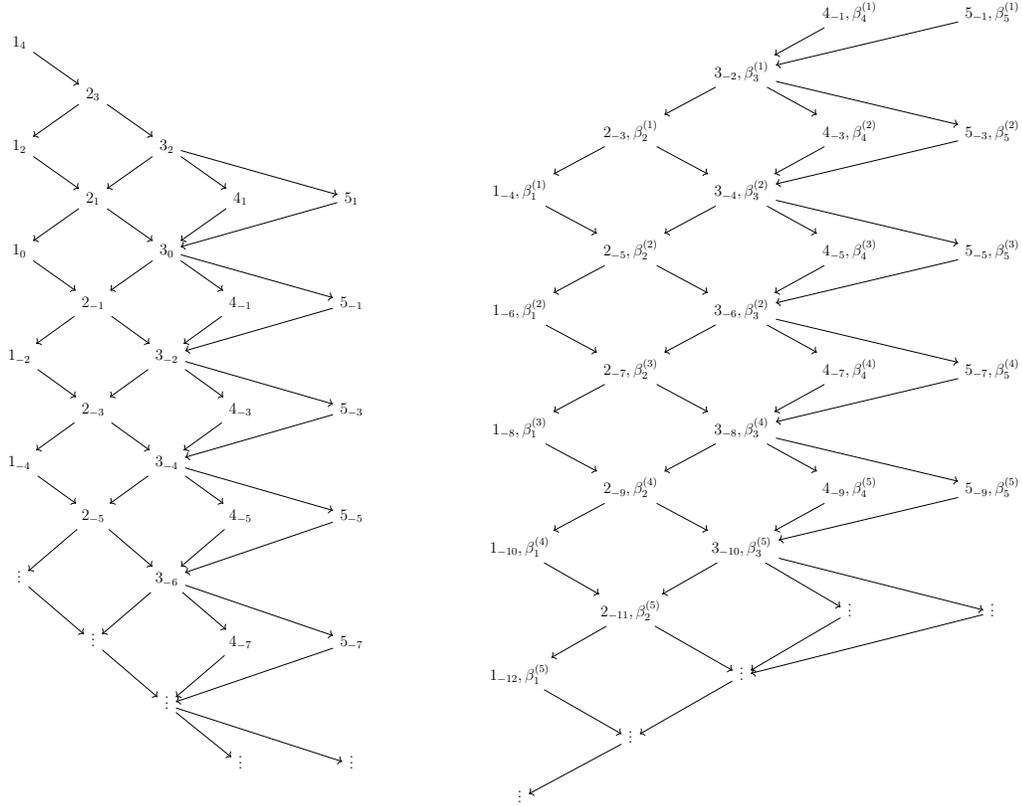

   We now introduce the following renormalizations: for every $(j,s)$ such that $\xistar(j) \geq s > \xistar(j)-2h+2$ and $1 \leq l < h$ we set 
\begin{align} \label{eq:def of Dprimejs}
{D'}_{j,s}^{(l)} := z_{j, \zeta(j,s,l)}^{-1} \mathbf{D} \left( X_{j,s}^{(l)} \right).
\end{align}    
The next Proposition shows that the ${D'}_{j,s}^{(l)}$ satisfy identities dual to the T-systems satisfied by the KR modules $X_{i,p}^{(k)} , 1 \leq k <h, \xi(i) \geq p > \xi(i)-2h+2$ in $K_0(\overline{\CQ})$. 


\begin{proposition} \label{prop : T-systems}
    Let $(j,s)$ such that $\xistar(j) \geq s > \xistar(j)-2h+2$ and let $1 \leq l < h$. Then we have that
\begin{align*}
  {D'}_{j,s-2}^{(l)} {D'}_{j,s}^{(l)} &= {D'}_{j,s-2}^{(l+1)} {D'}_{j,s}^{(l-1)} + \prod_{j' \sim j} {D'}_{j',s-1}^{(l)}, \quad l<h-1, \\
{D'}_{j,s-2}^{(h-1)} {D'}_{j,s}^{(h-1)} &=  {D'}_{j,s}^{(h-2)} + \prod_{j' \sim j} {D'}_{j',s-1}^{(h-1)}.
\end{align*}    
    \end{proposition}
    
\begin{proof}




We first consider the type $A_n$ case, for $j \in [1,n]$, the left hand side of the relation is
\begin{align*}
\frac{x_{j+m-1}}{x_{m-1}} {\bf D}\left( X_{j,s-2}^{(l)} \right) \frac{x_{j+m}}{x_{m}} {\bf D}\left( X_{j,s}^{(l)} \right). 
\end{align*}
Since $\xi^*(j-1) = \xi^*(j)-1$, $\xi^*(j+1) = \xi^*(j)+1$, we have that the right hand side of the relation is 
\begin{align*}
\frac{x_{j+m}}{x_{m}} {\bf D}\left( X_{j,s-2}^{(l+1)} \right) \frac{x_{j+m-1}}{x_{m-1}} {\bf D}\left( X_{j,s}^{(l-1)} \right) + \frac{x_{j+m-1}}{x_{m}} \frac{x_{j+m}}{x_{m-1}} {\bf D}\left( X_{j-1,s-1}^{(l)} \right)  {\bf D}\left( X_{j+1,s-1}^{(l)} \right),
\end{align*}
where we have used that $x_{n+m}=x_{m-1}$ in type $A_n$, and used that in the case of $l=h-1$, ${\bf D}(X_{j,s-2}^{(h)}) = z_{j, \zeta(j,s-2,h)} = z_{j, \zeta(j,s,h-1)} = \frac{x_m}{x_{j+m}}$.  
Therefore in this case, by (\ref{eq:apply bD to T-system on Qstar side which correspond to first two copies of CQ on Q side}), the left hand side of the relation for ${D'}_{j,s}^{l}$ is equal to the right hand side. Hence ${D'}_{j,s}^{(l)}$ satisfies T-system relations.

Now we consider type $D_n$ case. Let $j=n-1$. The left hand side of the relation is 
\begin{align*}
\frac{x_{\sigma^{m-1}(n-1)}}{x_{m-1}} {\bf D}\left( X_{n-1,s-2}^{(l)} \right) \frac{x_{\sigma^{m}(n-1)}}{x_{m}} {\bf D}\left( X_{n-1,s}^{(l)} \right).
\end{align*}
Since $\xi^*(n-2) = \xi^*(n-1)-1$ and $n-2+m \ge n-1$, we have that the right hand side of the relation is
\begin{align*}
\frac{x_{\sigma^m(n-1)}}{x_{m}} {\bf D}\left( X_{n-1,s-2}^{(l+1)} \right) \frac{x_{\sigma^{m-1}(n-1)}}{x_{m-1}} {\bf D}\left( X_{n-1,s}^{(l-1)} \right) + \frac{x_{n-1}x_n}{x_{m-1}x_m} {\bf D}\left( X_{n-2,s-1}^{(l)} \right),
\end{align*}
where we have used that in the case of $l=h-1$, we have that 
\[
{\bf D}\left( X_{n-1,s-2}^{(h)}\right) = z_{n-1, \zeta(n-1,s-2,h)} = z_{n-1, \zeta(n-1, s, h-1)} = z_{n-1, m} = \frac{x_m}{x_{\sigma^m(n-1)}}.  
\]
Therefore in this case, by (\ref{eq:apply bD to T-system on Qstar side which correspond to first two copies of CQ on Q side}), the left hand side of the relation for ${D'}_{j,s}^{l}$ is equal to the right hand side. Hence ${D'}_{j,s}^{(l)}$ satisfies T-system relations. The case of $j=n$ is proved in the same way. 

In type $D_n$, let $j=n-2$. The left hand side of the relation is
\begin{align*}
\frac{x_{n-1}x_n}{x_{m-1}x_{m-2}} {\bf D}\left( X_{n-2,s-2}^{(l)} \right) \frac{x_{n-1}x_n}{x_{m}x_{m-1}} {\bf D}\left( X_{n-2,s}^{(l)} \right). 
\end{align*}
Since $\xi^*(n-3) = \xi^*(n-2)-1$, $\xi^*(n-1) = \xi^*(n-2)+1$, $\xi^*(n) = \xi^*(n-2)+1$, the right hand side of the relation is
\begin{align*}
&  \frac{x_{n-1}x_n}{x_{m}x_{m-1}} {\bf D}\left( X_{n-2,s-2}^{(l+1)} \right) \frac{x_{n-1}x_n}{x_{m-1}x_{m-2}}  {\bf D}\left( X_{n-2,s}^{(l-1)} \right) \\
& +   \frac{x_{n-1}x_n}{x_{m-2}x_m} {\bf D}\left( X_{n-3,s-1}^{(l)} \right) \frac{x_{\sigma^{m-1}(n-1)}}{x_{m-1}} {\bf D}\left( X_{n-1,s-1}^{(l)} \right) \frac{x_{\sigma^{m-1}(n)}}{x_{m-1}} {\bf D}\left( X_{n,s-1}^{(l)} \right),
\end{align*}
where we have used that in the case of $l=h-1$, we have that 
\begin{align*}
{\bf D}\left( X_{n-2,s-2}^{(h)} \right) = z_{n-2,\zeta(n-2,s-2,h)} = z_{n-2,\zeta(n-2,s,h-1)} = z_{n-2, m} = \frac{x_{m-1}x_m}{x_{n-1}x_n}.  
\end{align*}
Therefore in this case, by (\ref{eq:apply bD to T-system on Qstar side which correspond to first two copies of CQ on Q side}), the left hand side of the relation for ${D'}_{j,s}^{l}$ is equal to the right hand side. Hence ${D'}_{j,s}^{(l)}$ satisfies T-system relations.

In type $D_n$, let $j \in [1,n-3]$. The left hand side of the relation is
\begin{align*}
\frac{x_{n-1}x_n}{x_{m-1}x_{m+j-n}} {\bf D}\left( X_{j,s-2}^{(l)} \right) \frac{x_{n-1}x_n}{x_{m}x_{m+j-n+1}} {\bf D}\left( X_{j,s}^{(l)} \right). 
\end{align*}
Since $\xi^*(j-1) = \xi^*(j)-1$, $\xi^*(j+1) = \xi^*(j)+1$, we have that the right hand side of the relation is 
\begin{align*}
\scalemath{0.86}{
\frac{x_{n-1}x_n}{x_{m}x_{m+j-n+1}} {\bf D}\left( X_{j,s-2}^{(l+1)} \right) \frac{x_{n-1}x_n}{x_{m-1}x_{m+j-n}} {\bf D}\left( X_{j,s}^{(l-1)} \right) + \frac{x_{n-1}x_n}{x_{m}x_{m+j-n+1}}  {\bf D}\left( X_{j-1,s-1}^{(l)} \right) \frac{x_{n-1}x_n}{x_{m-1}x_{m+j-n}} {\bf D}\left( X_{j+1,s-1}^{(l)} \right),}
\end{align*}
where we have used that in the case of $l=h-1$, we have that 
\begin{align*}
{\bf D}\left( X_{j,s-2}^{(h)} \right) = z_{j,\zeta(j,s-2,h)} = z_{j,\zeta(j,s,h-1)} = z_{j,m} = \frac{x_{m+j-n+1}x_m}{x_{n-1}x_n}.
\end{align*}
This concludes the proof of the Proposition.
\end{proof}


\begin{proof}[ \textbf{Proof of Theorem~\ref{thm : duality between KRs}}.]

Let $(i,p)$ such that $\xi(i) \geq p > \xi(i)-2h+2$ and $1 \leq k<h$. By \cite[Theorem 3.1]{HL16}, the class of the KR module $X_{i,p}^{(k)}$ is a Laurent polynomial in the classes $X_{k,t} , \xi(k) \geq t > \xi(k)-2h$ (recall the notation $X_{i,p}$ from~\eqref{eq : initial KRs}). Thus we can write 
$$ [X_{i,p}^{(k)}] = \sum_{ \underline{u}=(u_{j,t})} a_{\underline{u}} \prod_{j,t} [X_{j,t}]^{u_{j,t}}  $$
where $u_{j,t}$ are integers. On the other hand, if we denote by 
$$\Omega : K_0(\barCQ) \longrightarrow \mathbb{C}(\alpha_1, \ldots , \alpha_n)(x_1, \ldots , x_n) $$
the algebra homomorphism  determined by $\Omega : [X_{i,p}] \longmapsto D'_{i^{*} , \eta(p)}$,
then Proposition 4.12 (together with a straightforward induction) implies that 
$$ \Omega \left( [X_{i,p}^{(k)}] \right) = {D'}_{i^{*} , \eta(p)}^{(h-k)} $$
for all $1 \leq k < h$. In other words, we have 
$$ {D'}_{i^{*} , \eta(p)}^{(h-k)} = \sum_{ \underline{u}=(u_{j,t})} a_{\underline{u}} \prod_{j,t}  \left( {D'}_{j^{*},\eta(t)} \right)^{u_{j,t}}.  $$
 Thus ${D'}_{i^{*} , \eta(p)}^{(h-k)}$ can be expressed as a Laurent polynomial in the images under $\boldsymbol{D}$ of the classes of $\Sigma$, and this Laurent polynomial is the same as the Laurent expansion of $[X_{i,p}^{(k)}]$ in terms of the $[X_{j,t}]$ in $K_0(\overline{\CQ})$. Now, performing the specialization $x_i := (-1)^{d_i}$ for each $i \in I$, we get that if $X_{j^{*}, \eta(t)}$ is a class in $\Sigma_1$ then
 $${D'}_{j^{*},\eta(t)}|_{x_i = (-1)^{d_i}} = (-1)^{d_i}  \barDstar \left({X}_{j^{*},\eta(t)} \right) = \barD \left( X_{j,t} \right)  $$
 by Proposition~\ref{prop : duality for initial KRs}, whereas if $X_{j^{*}, \eta(t)}$ is a class in $\Sigma_2$ then 
 $${D'}_{j^{*},\eta(t)}  = 1 .  $$ 
 Thus we get that 
 $${D'}_{i^{*} , \eta(p)}^{(h-k)}|_{x_i = (-1)^{d_i}} =\sum_{ \underline{u}=(u_{j,t})} a_{\underline{u}} \prod_{j,t} \barD \left( X_{j,t} \right)^{u_{j,t}} = \barD \left( X_{i,p}^{(k)} \right) .  $$
 We now show that the specialization of ${D'}_{i^{*} , \eta(p)}^{(h-k)}$ at $x_i := (-1)^{d_i}$ coincides with $\epsilon_{i,p}^{(k)} \barDstar \left( X_{i^{*} , \eta(p)}^{(h-k)} \right)$.
 By Lemma 4.11, ${\bf D}  \left( X_{i^{*} , \eta(p)}^{(h-k)} \right)$ can be written as a Laurent polynomial in the KR classes in $\Sigma$. By definition, we have that $ {\bf D}([X]) = \barDstar([X])$ if $[X]$ is a KR class in $\Sigma_1$. For KR classes in $\Sigma_2$, we begin by noting that 
 $$ \zeta(j, s,l) = \frac{ s+ \xi(j^{*})+2l-2}{2}  = \frac{\xi(j^{*})-\eta(s)-2(h-l)+2}{2} $$
 which implies that
\begin{align} \label{eq : specializing z}
     z_{j,s, \zeta(j,s,l)}|_{x_i = (-1)^{d_i}}  = \epsilon_{j^{*} , \eta(s)}^{(h-l)}  
\end{align}
      for  every $(j,s)$ such that $\xistar(j) \geq s > \xistar(j)-2h+2$ and $1 \leq l \leq  h$. Then, using this fact with $l=h$, Proposition 4.4  implies that 
 $$ \barDstar \left( X_{j,s}^{(h)} \right) = \epsilon_{j^{*},\eta(s)}^{(0)} = z_{j, \zeta(j,s,h)} |_{x_i := (-1)^{d_i} , i \in I} = {\bf D} \left( X_{j,s}^{(h)} \right)|_{x_i := (-1)^{d_i}} $$
 for every $(j,s)$ such that $\xistar(j) \geq s > \xistar(j)-2h+2$. Therefore we obtain 
 $$ {\bf D}  \left( X_{i^{*} , \eta(p)}^{(h-k)} \right)|_{x_i := (-1)^{d_i}} = \barDstar \left( X_{i^{*} , \eta(p)}^{(h-k)} \right) . $$
 Now performing the specialization $x_i := (-1)^{d_i}$  in~\eqref{eq:def of Dprimejs}, using~\eqref{eq : specializing z} with $j=i^{*}, s = \eta(p)$ and $l=h-k$ we get 
 $${D'}_{i^{*} , \eta(p)}^{(h-k)}|_{x_i = (-1)^{d_i}} =\epsilon_{i,p}^{(k)} {\bf D}  \left( X_{i^{*} , \eta(p)}^{(h-k)} \right)|_{x_i := (-1)^{d_i}} =\epsilon_{i,p}^{(k)} \barDstar  \left( X_{i^{*} , \eta(p)}^{(h-k)} \right)  $$
 Putting everything together, we obtain 
 $$\barD \left( X_{i,p}^{(k)} \right) = {D'}_{i^{*} , \eta(p)}^{(h-k)}|_{x_i = (-1)^{d_i}} = \epsilon_{i,p}^{(k)} \barDstar  \left( X_{i^{*} , \eta(p)}^{(h-k)} \right) . $$



\end{proof}

\section{Proof of the main result}
\label{sec : proof of the main result}

In this section, we prove the main result Theorem \ref{thm : main thm intro}.

\subsection{Examples}
Before presenting the general proof, we provide several examples illustrating that $\tdxi \left( \chi_q(M) \right) = 0$ is highly non-trivial. In type $A_3$, take an orientation $1 \to 2 \to 3$ and a height function $\xi(1)=3$, $\xi(2)=2$, $\xi(3)=1$. Then $\tau=s_1s_2s_3$ and 
\begin{align*}
& \td_{\xi}(Y_{1,-1}) =
\frac{\alpha_2 (\alpha_1 + \alpha_2)}
     {\alpha_3 (\alpha_2 + \alpha_3) (\alpha_1 + \alpha_2 + \alpha_3)}, \quad 
\td_{\xi}(Y_{1,-3}) =
- \alpha_3 (\alpha_2 + \alpha_3) (\alpha_1 + \alpha_2 + \alpha_3), \\ 
& \td_{\xi}(Y_{2,0}) =
\frac{\alpha_1}
     {\alpha_2 (\alpha_2 + \alpha_3) (\alpha_1 + \alpha_2 + \alpha_3)}, \quad 
\td_{\xi}(Y_{2,-2}) =
\frac{\alpha_2 (\alpha_1 + \alpha_2) (\alpha_1 + \alpha_2 + \alpha_3)}
     {\alpha_3}, \\ 
& \td_{\xi}(Y_{2,-4}) =
\alpha_3 (\alpha_2 + \alpha_3), \quad 
\td_{\xi}(Y_{3,-1}) =
- \frac{\alpha_1 (\alpha_1 + \alpha_2) (\alpha_1 + \alpha_2 + \alpha_3)}
        {\alpha_2 (\alpha_2 + \alpha_3)}, \\ 
& \td_{\xi}(Y_{3,-3}) =
- \frac{\alpha_2 (\alpha_2 + \alpha_3)}
        {\alpha_3}.
\end{align*}
The $q$-character of $L(Y_{2,-4})$ is
\begin{align*}
\chi_q(L(Y_{2,-4})) = Y_{2,-4}
+ \frac{Y_{1,-3} Y_{3,-3}}{Y_{2,-2}}
+ \frac{Y_{3,-3}}{Y_{1,-1}}
+ \frac{Y_{1,-3}}{Y_{3,-1}}
+ \frac{Y_{2,-2}}{Y_{1,-1} Y_{3,-1}}
+ \frac{1}{Y_{2,0}}. 
\end{align*}
We check directly that $\td_{\xi}(\chi_q(L(Y_{2,-4})))=0$. 
In type $D_4$, take an orientation \tikz[baseline=-0.5ex,scale=0.8]{
\node (1) at (0,0) {$1$};
\node (2) at (0.8,0) {$2$};
\node (3) at (1.6,0.4) {$3$};
\node (4) at (1.6,-0.4) {$4$};
\draw[->] (1) -- (2);
\draw[->] (2) -- (3);
\draw[->] (2) -- (4);
} and a height function $\xi(1)=-1$, $\xi(2)=-2$, $\xi(3)=\xi(4)=-3$. Then $\tau=s_1s_2s_3s_4$ and 
\begin{align*}
\td_{\xi}(Y_{1,-3}) &=
\frac{\alpha_1}{\alpha_2 (\alpha_1 + \alpha_2)}, \\
\td_{\xi}(Y_{1,-5}) &=
\frac{\alpha_2}
     {\alpha_1 (\alpha_1 + \alpha_2 + \alpha_4) (\alpha_1 + \alpha_2 + \alpha_3)
      (\alpha_1 + \alpha_2 + \alpha_3 + \alpha_4) (\alpha_1 + 2\alpha_2 + \alpha_3 + \alpha_4)}, \\
\td_{\xi}(Y_{1,-7}) &=
-\,\frac{
\alpha_1 (\alpha_1 + \alpha_2) (\alpha_1 + \alpha_2 + \alpha_4) (\alpha_1 + \alpha_2 + \alpha_3)
(\alpha_1 + \alpha_2 + \alpha_3 + \alpha_4)
}{
\alpha_2 (\alpha_2 + \alpha_4) (\alpha_2 + \alpha_3) (\alpha_2 + \alpha_3 + \alpha_4)
}, 
\end{align*}
\begin{align*}
\td_{\xi}(Y_{2,-2}) &=
\frac{1}{\alpha_1 (\alpha_1 + \alpha_2)}, \\
\td_{\xi}(Y_{2,-4}) &=
\frac{1}{
\alpha_2 (\alpha_1 + \alpha_2) (\alpha_1 + \alpha_2 + \alpha_4)
(\alpha_1 + \alpha_2 + \alpha_3) (\alpha_1 + 2\alpha_2 + \alpha_3 + \alpha_4)
}, \\
\td_{\xi}(Y_{2,-6}) &=
\frac{
\alpha_1 (\alpha_1 + \alpha_2)
}{
(\alpha_2 + \alpha_4) (\alpha_2 + \alpha_3) (\alpha_2 + \alpha_3 + \alpha_4)
(\alpha_1 + \alpha_2 + \alpha_3 + \alpha_4)
(\alpha_1 + 2\alpha_2 + \alpha_3 + \alpha_4)
}, \\
\td_{\xi}(Y_{2,-8}) &=
\frac{
\alpha_2 (\alpha_1 + \alpha_2) (\alpha_1 + \alpha_2 + \alpha_4)
(\alpha_1 + \alpha_2 + \alpha_3) (\alpha_1 + 2\alpha_2 + \alpha_3 + \alpha_4)
}{
\alpha_4 \alpha_3 (\alpha_2 + \alpha_3 + \alpha_4)
}, 
\end{align*}
\begin{align*}
\td_{\xi}(Y_{3,-3}) &=
\frac{1}{
\alpha_1 (\alpha_1 + \alpha_2) (\alpha_1 + \alpha_2 + \alpha_3)
}, \\
\td_{\xi}(Y_{3,-5}) &=
\frac{
\alpha_1 (\alpha_1 + \alpha_2 + \alpha_3)
}{
\alpha_2 (\alpha_2 + \alpha_4) (\alpha_1 + \alpha_2 + \alpha_4)
(\alpha_1 + 2\alpha_2 + \alpha_3 + \alpha_4)
}, \\
\td_{\xi}(Y_{3,-7}) &=
\frac{
\alpha_2 (\alpha_2 + \alpha_4) (\alpha_1 + \alpha_2)
(\alpha_1 + \alpha_2 + \alpha_4)
}{
\alpha_3 (\alpha_2 + \alpha_3) (\alpha_2 + \alpha_3 + \alpha_4)
(\alpha_1 + \alpha_2 + \alpha_3)
(\alpha_1 + \alpha_2 + \alpha_3 + \alpha_4)
}, 
\end{align*}
\begin{align*}
\td_{\xi}(Y_{4,-3}) &=
\frac{1}{
\alpha_1 (\alpha_1 + \alpha_2) (\alpha_1 + \alpha_2 + \alpha_4)
}, \\
\td_{\xi}(Y_{4,-5}) &=
\frac{
\alpha_1 (\alpha_1 + \alpha_2 + \alpha_4)
}{
\alpha_2 (\alpha_2 + \alpha_3) (\alpha_1 + \alpha_2 + \alpha_3)
(\alpha_1 + 2\alpha_2 + \alpha_3 + \alpha_4)
}, \\
\td_{\xi}(Y_{4,-7}) &=
\frac{
\alpha_2 (\alpha_2 + \alpha_3) (\alpha_1 + \alpha_2)
(\alpha_1 + \alpha_2 + \alpha_3)
}{
\alpha_4 (\alpha_2 + \alpha_4) (\alpha_2 + \alpha_3 + \alpha_4)
(\alpha_1 + \alpha_2 + \alpha_4)
(\alpha_1 + \alpha_2 + \alpha_3 + \alpha_4)
}.
\end{align*}
The $q$-character of $L(Y_{2,-8})$ is 
\begin{align*}
& \chi_q(L(Y_{2,-8})) = Y_{2,-8} + \frac{Y_{1,-7} Y_{3,-7} Y_{4,-7}}{Y_{2,-6}}
+ \frac{Y_{3,-7} Y_{4,-7}}{Y_{1,-5}}
+ \frac{Y_{1,-7} Y_{4,-7}}{Y_{3,-5}}
+ \frac{Y_{1,-7} Y_{3,-7}}{Y_{4,-5}} \\
& + \frac{Y_{2,-6} Y_{4,-7}}{Y_{1,-5} Y_{3,-5}}
 + \frac{Y_{2,-6} Y_{3,-7}}{Y_{1,-5} Y_{4,-5}}
+ \frac{Y_{1,-7} Y_{2,-6}}{Y_{3,-5} Y_{4,-5}} + \frac{Y_{4,-7} Y_{4,-5}}{Y_{2,-4}}
+ \frac{Y_{2,-6}^2}{Y_{1,-5} Y_{3,-5} Y_{4,-5}} + \frac{Y_{3,-7} Y_{3,-5}}{Y_{2,-4}}
\\
& + \frac{Y_{1,-7} Y_{1,-5}}{Y_{2,-4}}+ \frac{Y_{4,-7}}{Y_{4,-3}}
+ \frac{Y_{2,-6}}{Y_{4,-5} Y_{4,-3}}
+ \frac{Y_{3,-7}}{Y_{3,-3}}
+ \frac{Y_{2,-6}}{Y_{3,-5} Y_{3,-3}} + 2 \frac{Y_{2,-6}}{Y_{2,-4}}
+ \frac{Y_{1,-5} Y_{3,-5} Y_{4,-5}}{Y_{2,-4}^2}
\\
& + \frac{Y_{1,-7}}{Y_{1,-3}}
+ \frac{Y_{2,-6}}{Y_{1,-5} Y_{1,-3}} + \frac{Y_{3,-5} Y_{4,-5}}{Y_{1,-3} Y_{2,-4}}
+ \frac{Y_{1,-5} Y_{4,-5}}{Y_{2,-4} Y_{3,-3}}
+ \frac{Y_{1,-5} Y_{3,-5}}{Y_{2,-4} Y_{4,-3}} + \frac{Y_{4,-5}}{Y_{1,-3} Y_{3,-3}}
\\
& + \frac{Y_{3,-5}}{Y_{1,-3} Y_{4,-3}}
+ \frac{Y_{1,-5}}{Y_{3,-3} Y_{4,-3}} + \frac{Y_{2,-4}}{Y_{1,-3} Y_{3,-3} Y_{4,-3}}
+ \frac{1}{Y_{2,-2}}.
\end{align*}
We check directly that $\td_{\xi}(\chi_q(L(Y_{2,-8})))=0$.

\subsection{Proof of the main result in the case \texorpdfstring{$Q = {Q_0^{op}}^{*}$}{Q = (Q0op)*}}

In this subsection, we prove that Theorem~\ref{thm : main thm intro} holds when $Q = Q_0$ is the monotonic orientation considered in Section~\ref{sec : duality}. For simplicity, we write $\td$ for $\td_{\xi}$ where $\xi$ is a height function adapted to $Q$.  We begin by showing that 
  $$ \td \left( \chi_q \left( L(Y_{i, \xi(i)-2h+2}) \right) \right) = 0  $$
  for every $i \in I$. By Theorem~\ref{thm : duality between KRs}, we have that 
  \begin{equation} \label{eq : duality in the proof of main result}
      \barD \left(  L(Y_{i, \xi(i)-2h+2}) \right) = \pm \barDstar \left( X_{j,\xistar(j)+2}^{(h-1)}  \right)
  \end{equation}
  where $j := i^{*}$. The class involved on the right hand side can be decomposed as a linear combination of classes of standard modules in $\CZ$, all of which are non-trivial (i.e. not isomorphic to the trivial representation). Indeed, recall from Section~\ref{sec : HL} that $\chi_q \left( X_{j,\xistar(j)+2}^{(h-1)} \right)$ is homogeneous with respect to the grading~\eqref{eq : grading}. Its degree is thus given by the dominant monomial, which here gives 
  $$ \beta_{j,\xistar(j)+2} + \cdots + \beta_{j, \xistar(j)+2h-2} =  \left( \sum_{k=0}^{h-2} \tau^k \right) (\beta_{j,\xistar(j)+2})   $$
  which is non zero because the order of $\tau$ is $h$. Hence all the standard modules involved in the decomposition of $X_{j,\xistar(j)+2}^{(h-1)}$ have $q$-characters of non-zero degree with respect to~\eqref{eq : grading} and hence are not isomorphic to the trivial representation.  Furthermore, each of these standard modules is a tensor product of fundamental representations at least one of which is of the form  $L(Y_{k,r})$ with $r > \xistar(k)$. 
  In view of the definition of $\barDstar$, it follows that the right hand side of~\eqref{eq : duality in the proof of main result} is zero.
  
   On the other hand, the truncated $q$-character of $L(Y_{i, \xi(i)-2h+2})$ coincides with its entire $q$-character. Indeed, the lowest $l$-weight monomial is $Y_{i^{*},p}$ where $p = \xi(i)-h+2 < \xi(i)$ and thus all the terms of the $q$-character of $L(Y_{i, \xi(i)-2h+2})$ belong to $\YxiZ$ (cf. Section~\ref{sec : HL}). Hence the left hand side of~\eqref{eq : duality in the proof of main result} is actually $\td \left( \chi_q  \left( L(Y_{i, \xi(i)-2h+2}) \right) \right)$ which proves the desired result in the case of the fundamental representations $L(Y_{i, \xi(i)-2h+2}) , i \in I$. We claim this implies that the same statement holds for all fundamental representations in $\CZ$. Indeed, recall from Section~\ref{sec : HL} that for every $(i,p) \in \IZ$, $Y_{i,p}^{-1} \chi_q \left( L(Y_{i,p}) \right)$ is  a polynomial (with constant term $1$) in the $A_{j,s}^{-1} , (j,s-1) \in \IZ$ such that, given $p,p' \in i + 2 \mathbb{Z}$, $Y_{i,p'}^{-1} \chi_q \left( L(Y_{i,p'}) \right)$ is obtained from $Y_{i,p}^{-1} \chi_q \left( L(Y_{i,p}) \right)$ by performing the substitution $A_{j,s} := A_{j,s+p'-p}$ for all $(j,s)$. Consequently, Corollary~\ref{coro : Dtilde of A inverse} implies that 
  $$ \td \left(Y_{i,p}^{-1} \chi_q \left( L(Y_{i,p}) \right) \right) =\td \left(Y_{i,p'}^{-1} \chi_q \left( L(Y_{i,p'}) \right) \right) |_{\alpha_k := \tau^{ \frac{p'-p}{2}}(\alpha_k)}. $$
  As $\tau$ acts as an automorphism of $\mathfrak{t}^{*}$, this implies that if $\td \left(  \chi_q \left( L(Y_{i,p}) \right) \right)=0$ for one particular $p \in i + 2 \mathbb{Z}$, then the same will hold for all $p' \in i + 2 \mathbb{Z}$. This concludes the proof.

 \subsection{Arbitrary orientation}

 In this section we prove that if the statement 
 $$ \forall j,s \quad \tdxi(\chi_q(L(Y_{j,s})))=0 $$
  holds for a height function $\xi$ adapted to an orientation $Q$, then  it also holds for any height function adapted to any arbitrary orientation. Recall that for a quiver $Q$ and a vertex $i$ of $Q$, we denote by $\gamma_i^{Q}$ the dimension vector of the indecomposable injective $\mathbb{C}Q$-module at $i$; furthermore, we let $\mathbf{s}_iQ$ denote the quiver  defined by reversing all the arrows adjacent to $i$ in $Q$.  Given that any Dynkin quiver can be obtained from a given one by performing a certain sequence of transformations $\mathbf{s}_i$ where $i$ is a source of the quiver obtained at each step, it is enough to prove that if the statement holds for an orientation $Q$, then it holds for all orientations $\mathbf{s}_iQ$ such that $i$ is a source of $Q$. This is achieved by establishing the following:

   \begin{proposition}
  Let $Q$ be an arbitrary orientation of the Dynkin diagram of $\mathfrak{g}$ and let $i$ be a source of $Q$.  Then denoting by $s_i$ the simple reflections of the Weyl group of $\mathfrak{g}$, we have 
  $$ \forall j,s \quad  \td_{\mathbf{s}_i Q}(Y_{j,s})(\alpha_1, \ldots , \alpha_n) = (-\alpha_i)^{-(\beta_{j,s},\alpha_i)} \td_{Q}(Y_{j,s})(s_i \alpha_1, \ldots, s_i \alpha_n) . $$

   \end{proposition}

 \begin{proof}
 The proof is based on the following classical facts. Fix a height function $\xi$ adapted to $Q$. Then $\mathbf{s}_i \xi := \xi - 2 \delta_{i,j}$ is a height function adapted to $\mathbf{s}_iQ$, and moreover the following hold:
  \begin{itemize}
      \item For all $j,s$ we have that $\beta_{j,s}^{\mathbf{s}_i \xi} = s_i(\beta_{j,s}^{\xi})$.
      \item Denoting by $\varphi_{\xi}$ (resp. $\varphi_{\mathbf{s}_i\xi}$) the bijection $I \times \mathbb{Z} \longrightarrow \Delta \times \mathbb{Z}$ induced by $\xi$ (resp. $\mathbf{s}_i \xi$), we have that 
      $$ \forall j \neq i \quad \gamma_j^Q = \gamma_j^{\mathbf{s}_i Q} \quad \text{and} \quad  \varphi_{\xi}^{-1}(\gamma_j^Q,0) = \varphi_{\mathbf{s}_i \xi}^{-1}(\gamma_j^{\mathbf{s}_i Q},0) = (j, \xi(j)) $$
      $$ \varphi_{\mathbf{s}_i \xi}^{-1}(\gamma_i^{\mathbf{s}_i Q},0) = (i, \xi(i)-2) \qquad \varphi_{\mathbf{s}_i \xi}^{-1}(\alpha_i,0) = (i^{*}, \xi(i)-h) . $$
  \end{itemize}
  In other words, applying the elementary transformation $s_iQ$ has little impact on the Auslander-Reiten combinatorics of $Q$: the only place where there are changes are at the vertices corresponding to the respective indecomposable injectives at $i$ for $Q$ (which is $\alpha_i$ given that $i$ is a source in $Q$) and $s_iQ$. 
 This implies that if $j,s$ is such that $s \leq (s_i\xi)(j)$, then one has 
 \begin{align*}
 \td_{\mathbf{s}_i \xi}(Y_{j,s}) &= \prod_{t \leq (s_i\xi)(k)} (s_i(\beta_{k,t}))^{\widetilde{C}_{j,k}(t-s+1) - \widetilde{C}_{j,k}(t-s-1)} \\
 &= (- \alpha_i)^{\widetilde{C}_{j,i}(\xi(i)-s-1) - \widetilde{C}_{j,i}(\xi(i)-s+1)} \prod_{t \leq \xi(k)} (s_i(\beta_{k,t}))^{\widetilde{C}_{j,k}(t-s+1) - \widetilde{C}_{j,k}(t-s-1)} \\
 &=  (-\alpha_i)^{-(\beta_{j,s},\alpha_i)} \td_{Q}(Y_{j,s})(s_i \alpha_1, \ldots, s_i \alpha_n) .
 \end{align*}
 By periodicity of $\td_{\xi}$ and $\td_{\mathbf{s}_i \xi}$, this can be extended to all $j,s$. Thus the proposition is proved.
 \end{proof}

 We also get a simple relationship for the variables $A_{i,p}$, namely
  $$ \td_{\mathbf{s}_i Q}(A_{j,s})(\alpha_1, \ldots , \alpha_n) =  \td_{Q}(A_{j,s})(s_i \alpha_1, \ldots, s_i \alpha_n) . $$
 In particular, as the $q$-character of any fundamental representation can be written as $Y_{j,s}$ times a polynomial in the variables $A_{i,p}^{-1}$, we get that 
 $$ \forall j,s \quad  \td_{\mathbf{s}_i Q}(\chi_q(L(Y_{j,s})))(\alpha_1, \ldots , \alpha_n) = (-\alpha_i)^{-(\beta_{j,s},\alpha_i)} \td_{Q}(\chi_q(L(Y_{j,s})))(s_i \alpha_1, \ldots, s_i \alpha_n). $$
 As the linear transformation $s_i$ is invertible,  we conclude that 
 $$  \td_{\xi}(\chi_q(L(Y_{j,s})))=0 \quad   \Rightarrow \quad  \td_{\mathbf{s}_i \xi}(\chi_q(L(Y_{j,s})))=0   $$
 for all $j,s$.

\subsection{Exceptional types}

In type $E_6, E_7, E_8$, we start from the initial seed and assign $\td$ to the cluster variables and frozen variables in the seed. We apply mutation sequences which produce Kirillov--Reshetikhin modules (in particular, fundamental modules) to the initial seed. We have checked that $\td$ of full $q$-characters of fundamental modules are all zero. This implies that for any non-trivial standard module $M$ in $\CZ$ of type $E_6, E_7, E_8$, we have that $\tdxi \left( \chi_q(M) \right) = 0$. The SageMath codes for these computations could be found in: 

\url{https://github.com/lijr07/D-tilde-of-fundamental-modules-in-type-E}. 

\bibliographystyle{alpha}

\end{document}